\documentclass[11pt]{amsart}
\usepackage{mathdots}
\usepackage{amssymb}
\usepackage{tikz-cd}
\usepackage{hyperref}
\usepackage{mathtools}

\newtheorem{theorem}{Theorem}[section]

\newtheorem{proposition}[theorem]{Proposition}
\newtheorem{conjecture}[theorem]{Conjecture}
\newtheorem{corollary}[theorem]{Corollary}
\theoremstyle{definition}
\newtheorem{definition}[theorem]{Definition}
\newtheorem{example}[theorem]{Example}

\theoremstyle{remark}

\usepackage[margin=2.5cm]{geometry}
\numberwithin{equation}{section}

\newcommand{\wt}{\widetilde}

\newcommand{\calO}{\mathcal{O}}

\newcommand{\calV}{\mathcal{V}}

\newcommand{\bbA}{\mathbb{A}}

\newcommand{\bbC}{\mathbb{C}}

\DeclareMathOperator{\GL}{GL}

\DeclareMathOperator{\Tr}{Tr}

\DeclareMathOperator{\SL}{SL}
\DeclareMathOperator{\SO}{SO}
\DeclareMathOperator{\Sp}{Sp}

\DeclareMathOperator{\Hom}{Hom}

\DeclareMathOperator{\diag}{diag}
\DeclareMathOperator{\Mat}{Mat}
\DeclareMathOperator{\Rank}{Rank}
\DeclareMathOperator{\Ind}{Ind}
\DeclareMathOperator{\res}{res}
\DeclareMathOperator{\ind}{ind}
\DeclareMathOperator{\Res}{Res}

\begin{document}

\title{ Theta liftings of non-generic representations on double covers of orthogonal groups}

\author{Yusheng Lei}
\address{Department of Mathematics, Boston College,
	Chestnut Hill, MA-02467}
\email{yusheng.lei@bc.edu}

\subjclass[2021]{}

\date{\today.}



\begin{abstract}
We study the generalized theta lifting between the double covers of split special orthogonal groups, which uses the non-minimal theta representations constructed by Bump, Friedberg and Ginzburg. We focus on the theta liftings of non-generic representations and make a conjecture that gives an upper bound of the first non-zero occurrence of the liftings, depending only on the unipotent orbit. We prove both global and local results that support the conjecture. 
\end{abstract}

\maketitle


.
\tableofcontents
\section{\textbf{Introduction}}
The classical theta correspondence, based on the Weil representation, has been much studied. Crucially, the Weil representation is a  minimal representation. This property plays a key role in establishing many properties of the classical theta correspondence. In a small number of cases, there are also non-minimal theta correspondences. Let $F$ be a number field containing the group of fourth roots of unity, with the ring of adeles $\bbA$. In \cite{BumpFG03}, Bump-Friedberg-Ginzburg constructed the global theta representation $\Theta_{m}$ on the double cover $\wt{\SO}_{m}(\bbA)$ of the split odd orthogonal group $\SO_{m}(\bbA)$ as the residues of certain metaplectic Eisenstein series. In contrast to the Weil representation, such a theta representation is not minimal. However,  its Fourier coefficients attached to most unipotent orbits vanish. This allows the construction of non-minimal theta liftings  \cite{BumpFG06}. Specifically, suppose $(\pi,\calV)$ is an irreducible cuspidal genuine automorphic representation of $\wt{\SO}_{2k+1}(\bbA)$. Let $\SO_{2k'}$ be a split even orthogonal group. The embedding  $\SO_{2k'}(\bbA)\times \SO_{2k+1}(\bbA)\hookrightarrow\SO_{2k+2k'+1}(\bbA)$ is covered by the embedding of $\wt{\SO}_{2k'}(\bbA) \times \wt{\SO}_{2k+1}(\bbA)$ into $\wt{\SO}_{2k+2k'+1}(\bbA)$. Consider the theta representation $\Theta_{2k+2k'+1}$ on $\wt{\SO}_{2k+2k'+1}(\bbA)$. For any $\varphi\in \calV$ and $\theta_{2k+2k'+1}$ a function in the representation space of $\Theta_{2k+2k'+1}$, Bump-Friedberg-Ginzburg defined a function $f$ on $\wt{\SO}_{2k'}(\bbA)$ (see equation (2) of \cite{BumpFG06}) via the integral
\begin{equation}
	f(h) = \int_{\SO_{2k+1}(F)\setminus \SO_{2k+1}(\bbA)} \varphi(g)\bar{\theta}_{2k+2k'+1}(h,g)dg.
\end{equation}
Functions of the form $f(h)$ generate a genuine automorphic representation $\Theta_{2k+2k'+1}(\pi)$ on the cover $\wt{\SO}_{2k'}(\bbA)$.



By fixing the representation $\pi$ and using the theta representations $\wt{\SO}_{2k+2k'+1}(\bbA)$ with varying $k'$, one obtains a tower of liftings of the representation $\pi$ to the groups $\wt{\SO}_{2k'}(\bbA)$. According to \cite{BumpFG06}, for a fixed genuine cuspidal automorphic representation $\pi$ on $\wt{\SO}_{2k+1}(\bbA)$, one has the  following:
\begin{enumerate}
	\item  As an automorphic representation of $\wt{\SO}_{8k}(\bbA)$, $\Theta_{10k+1}(\pi) \neq 0$.
	\item  If $\Theta_{2k+2k'+1}(\pi)=0$, then  $\Theta_{2k+2k'-1}(\pi)=0$.
\end{enumerate}
In view of these, it is natural to ask when the first non-zero lifting occurs along the tower.  In \cite{BumpFG06}, Bump-Friedberg-Ginzburg show that if $\Theta_{4k+5}(\pi)$ is generic as an automorphic representation of $\wt{\SO}_{2k+4}(\bbA)$, then the representation $\pi$ of $\wt{\SO}_{2k+1}(\bbA)$ must be generic as well. 
They also make a conjecture that a generic representation $\pi$ on $\wt{\SO}_{2k+1}(\bbA)$ should lift to a generic representation on $\wt{\SO}_{2k+4}(\bbA)$.


However, little is known if the representations are not generic, i.e. not supported on the maximal unipotent orbit. In this paper, we make a general conjecture on when the lift of a given automorphic representation of $\wt{\SO}_{2k+1}(\bbA)$ is nonzero, depending only on the unipotent orbit that the representation is supported on. Recall that unipotent orbits are parametrized by partition of integers. We require that the representation is supported on a unipotent orbit whose corresponding partition consists of only odd integers. This condition implies that the attached unipotent subgroup $V_{2,O}$ defined in \hyperref[sec3]{Section \ref*{sec3}} below is the unipotent radical of a parabolic subgroup.
\begin{conjecture}\label{cj1}
	Let $(\pi,\calV)$ be an irreducible cuspidal genuine automorphic representation of  $\wt{\SO}_{2k+1}(\bbA)$. Suppose $\pi$ is supported on the unipotent orbit
	\[
	\calO = \left((2n_1+1)^{r_1}(2n_2+1)^{r_2}\cdots(2n_p+1)^{r_p}\right) 
	\]
	with $n_1>n_2>\cdots >n_p\geqslant 0$ and $r_i>0$ for all $i$. Let $l= r_1+r_2+\cdots +r_p$ be the length of the partition corresponding to $\calO$. Then $\pi$ lifts nontrivially to an automorphic representation $\Theta_{4k+2l+3}(\pi)$ of $\wt{\SO}_{2k+2l+2}(\bbA)$ which is supported on the unipotent orbit
	\[
	\calO' =  \left((2n_1+3)^{r_1}(2n_2+3)^{r_2}\cdots(2n_p+3)^{r_p}(1)\right).\\
	\]
\end{conjecture}
\hyperref[cj1]{Conjecture \ref*{cj1}} gives an upper bound of the first non-zero occurrence of the theta lifting. In the generic case where $\calO = (2k+1)$, \hyperref[cj1]{Conjecture \ref*{cj1}} agrees with the conjecture made in \cite{BumpFG06} and mentioned above. In \hyperref[deqp]{Proposition \ref*{deqp}}, we show that this conjecture is consistent with the ``dimension equation" described in \cite{GIN06}, \cite{GIN14} and \cite{FG19}, which proposes dimension constraints on when the first non-zero lifting may occur. We remark that not every orbit of $\SO_{2k+1}$ has all odd parts, but we do not know what to expect when there are even parts in the partition.

In this paper, we prove the following theorem which gives evidence towards the above conjecture. 
\begin{theorem}\label{mt}
		Let $(\pi, \calV)$ be an irreducible cuspidal genuine automorphic representation of $\wt{\SO}_{2k+1}(\bbA)$. Suppose the theta lifting $\Theta_{4k+2l+3}(\pi)$, as a representation of $\wt{\SO}_{2k+2l+2}(\bbA)$, has a non-zero Fourier coefficient associated with the unipotent orbit 
	\[
	\calO' =\left((2n_1+3)^{r_1}(2n_2+3)^{r_2}\cdots(2n_p+3)^{r_p}(1)\right).
	\]
	Then the representation $\pi$ has a non-zero Fourier coefficient associated with the unipotent orbit
	\[
	\calO =\left((2n_1+1)^{r_1}(2n_2+1)^{r_2}\cdots(2n_p+1)^{r_p}\right).
	\]
\end{theorem}
\noindent In the generic case, \hyperref[mt]{Theorem \ref*{mt}} agrees with the result proved in \cite{BumpFG06} and mentioned above. 

Moveover, we establish a local counterpart of \hyperref[mt]{Theorem \ref*{mt}} (which is new even in the generic case). In the local setting, we turn our attention to the category of genuine admissible representations on the double covers of the split special orthogonal groups over a non-archimedean local field $F$.  In \cite{BumpFG03}, the local theta representation on the double cover of a split odd orthogonal group $\wt{\SO}_{2k+1}(F)$ is constructed as the image of an intertwining operator. Fourier coefficients as the global analytic tool are replaced by the twisted Jacquet modules. We prove the following result:
\begin{theorem}
	Let $(\pi,\calV)$ be an irreducible genuine admissible representation of $\wt{\SO}_{2k+1}(F)$. Suppose there exists an irreducible admissible representation $\Theta(\pi)$ of $\wt{\SO}_{2k+2l+2}(F)$ such that, as representations of the group $\wt{\SO}_{2k+1}(F)\times \wt{\SO}_{2k+2l+2}(F)$,
	\begin{equation}
		\Hom_{\wt{\SO}_{2k+1}\times \wt{\SO}_{2k+2l+2}}\left(\Theta_{4k+2l+3},\pi\otimes \Theta(\pi)\right) \neq 0.	\end{equation}
	Furthermore, suppose there exists a non-trivial character $\psi_{\calO'}$ (explicitly defined in \hyperref[sec6]{Section \ref*{sec6}}) associated with the unipotent orbit $\calO'$ such that the corresponding twisted Jacquet module of $\Theta(\pi)$  is non-zero. Then there exists a non-trivial character associated with the unipotent orbit $\calO$ such that the corresponding twisted Jacquet module of $\pi$ is also non-zero.
\end{theorem}

In the case of the classical symplectic-orthogonal theta liftings based on the Weil representation, Ginzburg-Gurevich \cite{GinG05} give both upper and lower bounds for the first non-zero occurrence in the theta tower. These bounds can be parametrized by the partition corresponding to the unipotent orbit that supports the cuspidal automorphic representation of $\Sp_{2k}(\bbA)$.  

The liftings considered here are related to the extension of Langlands functoriality to covering groups as follows. According to \cite{McN09}, \cite{Wei11} and \cite{Sa88}, one can define the dual group of a metaplectic group. In the case of the metaplectic double cover of $\SO_m$, we have that 
\[
^{L}\wt{\SO}_{m}^0\cong \SO_m(\bbC).
\]
This suggests that there should be a lifting of genuine automorphic representations from $\wt{\SO}_{2k+1}$ to $\wt{\SO}_{2k'}$ corresponding to the inclusion of $\SO_{2k+1}(\bbC)$ into $\SO_{2k'}(\bbC)$ with $k'>k$.

The study of non-generic cuspidal automorphic representations is an important part of understanding the automorphic
discrete spectrum. Jiang \cite{Jiang14} proposed a conjecture that relates Arthur parameters to the maximal unipotent orbit that supports an automorphic representation. In Section 13 of \cite{LE17}, Leslie conjectured an extension of Arthur parameters to the metapletic groups. In view of these works, the results of this paper are conjecturally related to the question of how Arthur parameters behave under the non-minimal theta liftings introduced in \cite{BumpFG06}. 

This paper is organized as follows: After setting up the basic notations, we briefly recall the construction of the metaplectic double cover of the split orthogonal groups in \hyperref[sec2]{Section \ref*{sec2}}. In \hyperref[sec3]{Section \ref*{sec3}}, we review the definition of the unipotent orbits and recall the construction of the Fourier coefficients and the twisted Jacquet modules associated to a unipotent orbit. These are the global and local tools for proving the respective main theorems. In \hyperref[sec4]{Section \ref*{sec4}}, we briefly recall the construction of both the local and global theta representations of the double cover $\wt{\SO}_{2k+1}$. We then prove an invariance property of the theta representations which is crucial for the proof of the main theorem. We also establish the compatibility of \hyperref[cj1]{Conjecture \ref*{cj1}} with the dimension equation.  In \hyperref[sec5]{Section \ref*{sec5}}, we prove the global main theorem \hyperref[mt]{Theorem \ref*{mt}}. Lastly, the local theory is treated in \hyperref[sec6]{Section \ref*{sec6}}.\\

\noindent \textbf{Acknowledgement.} I would like to express my gratitude to my advisor Solomon Friedberg for suggesting this research topic, as well as for providing a tremendous amount of advice and support. I would also like to thank Yuanqing Cai, Stella Sue Gastineau, David Ginzburg and Hao Li for many helpful discussions. 

\section{\textbf{Preliminaries}}\label{sec2}
\subsection{Split orthogonal groups}
In this paper, we let $F$ be either a number field with ring of adeles $\bbA$ or a non-archimedean local field of residue characteristic not equal to $2$.  Fix an algebraic closure $\bar{F}$ of $F$, and denote by 
\[
\mu_4=\{x\in \bar{F} :x^4=1\}
\]the group of all forth roots of unity in $\bar{F}$. Throughout this paper, we assume that $F$ contains $\mu_4$. We also fix a choice of non-trivial additive character by $\psi: F\setminus \bbA \rightarrow \bbC^\times$ when $F$ is a number field, or by $\psi:  F \rightarrow\bbC^\times$ which is unramified when $F$ is a non-archimedean local field.

For any positive integer $m$, let $\SO_{m}(F)$ denote the split special orthogonal group consisting of $g\in \SL_{m}$ such that $gJ_mg^T =J_m$, where
\[
J_m = \begin{pmatrix}
&&&&1\\
&&&1&\\
&&\iddots&&\\
&1&&&\\
1&&&&
\end{pmatrix} \in \Mat_{m\times m}(F).
\]
The maximal unipotent subgroup  of $\SO_{m}$ contains $n = \lfloor m/2 \rfloor$ positive simple roots. Let $e_{i,j}$ be the $m\times m$ matrix with value one on the $(i,j)$-th entry and zero elsewhere. We denote by $\alpha_i$ ($1\leqslant i\leqslant n$) the positive simple roots with respect to the usual order in the standard Borel subgroup of upper triangular matrices, and we let each of the corresponding one-parameter subgroups be $r\mapsto x_{\alpha_i}(r)$, where
\[
x_{\alpha_i}(r)= \exp\left(r(e_{i,i+1}-e_{n-i,n-i+1})\right)
\]
when $m$ is odd, and
\[
x_{\alpha_i}(r)=\begin{cases}
\exp\left(r(e_{i,i+1}-e_{n-i,n-i+1})\right)&\text{if}\,1\leqslant i<n\\
\exp\left(r(e_{n-1,n+1}-e_{n,n+2})\right)&\text{if}\ i =n
\end{cases}
\]
when $m$ is even. 

We fix an embedding of any two orthogonal groups $\SO_{2k+1}$ and $\SO_{2k'}$ into $\SO_{2k+2k'+1}$ by
\begin{equation}\label{ebd}
	\iota(h,g) =\begin{pmatrix}
		a&0&b\\
		0&g&0\\
		c&0&d
	\end{pmatrix}\in \SO_{2k+2k'+1},\quad g\in\SO_{2k+1},h= \begin{pmatrix}
		a&b\\c&d
	\end{pmatrix}\in\SO_{2k'}.
\end{equation}

\subsection{The double cover of the split orthogonal group $\SO_m$} 
Suppose $F$ is a non-archimedean local field. Let $\wt{\SL}_m(F)$ be the metaplectic $4$-fold cover defined in \cite{Mat69} and \cite{KP84}, with the corresponding cocycle denoted by $\sigma$. This covering group satisfies the short exact sequence 
\begin{equation}
1\rightarrow \mu_4\rightarrow \wt{\SL}^{(\sigma)}_m(F)\rightarrow \SL_m(F)\rightarrow 1.
\end{equation}
Pulling back the image of $\SO_m(F)$ in $\SL_{m}(F)$, we obtain a central extension $\wt{\SO}_m(F)$ of $\SO_m(F)$. According to \cite{BumpFG00} and \cite{BumpFG03}, the square of the cocycle $\sigma$ restricted to $\SO_m(F)$ is almost trivial. Hence, this gives a double cover of the split orthogonal group $\SO_m(F)$.

The same construction goes through in the global situation as well. Suppose $F$ is a number field with its ring of adeles $\bbA$. We denote the double cover of $\SO_m(\bbA)$ by $\wt{\SO}_m(\bbA)$. 

\begin{proposition}\label{p21}
	The discrete subgroup $\SO_m(F)$ and the unipotent radical $N(\bbA)$ of the upper triangular Borel subgroup of $\SO_m(\bbA)$ split in the double cover $\wt{\SO}_m(\bbA)$. 
\end{proposition}
\begin{proof}
	See Section 3 of \cite{Ka17}.
	\end{proof}
\section{\textbf{Fourier coefficients associated to a unipotent orbit}}\label{sec3}
In this section, we recall the connection between Fourier coefficients and unipotent orbits of the odd orthogonal group $\SO_{2k+1}$. See \cite{CM93}, \cite{GIN06} and \cite{GIN14} for more details.

\subsection{Unipotent orbits}
Let $F$ be either a number field or a non-archimedean local field, with a fixed algebraic closure $\bar{F}$. Unipotent orbits of the group $\SO_{2k+1}$ are parametrized by partitions of $2k+1$ with the restriction that each even number occurs with even multiplicity. For an orbit $\calO$ corresponding to the partition $(p_1^{r_1}p_2^{r_2}\cdots p_s^{r_s})$ where $p_i> p_{i+1}$ and $r_i>0$ for all $i$, we write 
\[
\calO = (p_1^{r_1}p_2^{r_2}\cdots p_s^{r_s}).
\]
Define the length of the partition to be $l = r_1+r_2+\cdots +r_s$.

 Suppose $\calO_1 =(p_1p_2\cdots p_r)$ and $\calO_2 =(q_1q_2\cdots q_s)$. We impose a partial order by  $\calO_1 \geqslant \calO_2$ if $p_1+\cdots+p_i \geqslant q_1+\cdots +q_i$ for all $1\leqslant i \leqslant s$.

Let $\calO= (p_1^{r_1}p_2^{r_2}\cdots p_s^{r_s})$. For each $p_i$, we associate $r^i$ copies of the torus element 
\[
h_{p_i}(t) = \diag(t^{p_i-1}, t^{p_i-3}, \cdots, t^{3-p_i}, t^{1-p_i} ).
\]
We obtain a one parameter torus element $h_{\calO}(t)$ with  non-increasing powers of $t$ along the diagonal after combining and rearranging all the $h_{p_i}(t)$'s. For example, if $\calO =(3^21)$, then
\[
h_{\calO}(t) = \diag(t^2,t^2,1,1,1,t^{-2},t^{-2}).
\]
 
The conjugation action of $h_\calO(t)$ on the unipotent radical $N$ of the upper triangular Borel subgroup $B$ of $\SO_{2k+1}$ induces a filtration on $N$
\[
I_{2k+1} \subset \cdots \subset V_{2,\calO} \subset V_{1,\calO}\subset V_{0,\calO}=N,
\]
where 
\[
V_{i,\calO}= \{x_{\alpha}(r)\in N : h_\calO(t)x_\alpha(r)h_\calO(t)^{-1} =x_\alpha(t^jr) \quad \text{for some}\quad j\geqslant i\}.
\]
Define
\[
M(\calO) = T\cdot \{x_{\pm\alpha}(r):  h_\calO(t)x_\alpha(r)h_\calO(t)^{-1} =x_\alpha(r)\}.
\]
Then $P(\calO) = M(\calO)V_{1,\calO}$ is a standard maximal parabolic subgroup of $\SO_{2k+1}$.

We say that a unipotent orbit $\calO= (p_1^{r_1}p_2^{r_2}\cdots p_s^{r_s})$ is odd if all the integers $p_i$ are odd. In \hyperref[gfc]{Section \ref*{gfc}}, we will see that the Fourier coefficients associated to the unipotent orbit $\calO$ are given by integration against $V_{2,\calO}$. In general, $V_{2,\calO}$ is not the unipotent radical of a parabolic subgroup. However, if $\calO$ is odd, then $V_{2,\calO}=V_{1,\calO}$ is the unipotent radical of the parabolic subgroup $P(\calO)$.

 Let $V_{2,\calO}^{(1)}$ be the commutator subgroup of $ V_{2,\calO}$. Over $\bar{F}$, the Levi subgroup $M(\calO)$ acts by conjugation on the maximal abelian quotient $V_{2,\calO}/V_{2,\calO}^{(1)}$ with a dense open orbit. Pick a representative $u_0$ of this orbit, and set $M^{u_0}(\calO)(\bar{F})$ to be its stabilizer. Although the group $M^{u_0}(\calO)(\bar{F})$ depends on the choice of $u_0$, its Cartan type is independent of the choice.

\subsection{Generic characters} Suppose $F$ is a number field with  ring of adeles $\bbA$. Let $L_{2,\calO} =V_{2,\calO}/V_{2,\calO}^{(1)}$ be the maximal abelian quotient of $V_{2,\calO}$. The action by conjugation of $M(\calO)$ on $V_{2,\calO}$ induces an action of $M(\calO)(F)$ on the character group
\[
\widehat{L_{2,\calO}(F)\setminus L_{2,\calO}(\bbA)}\cong L_{2,\calO} (F).
\]  

We call $\psi_{\calO}: L_{2,\calO}(F)\setminus L_{2,\calO}(\bbA) \rightarrow \bbC^\times$ a generic character if the connected component of its stabilizer in $M(\calO)(F)$ is of the same type as $M^{u_0}(\calO)(\bar{F})$. We extend any such character trivially to $V_{2,\calO}(F)\setminus V_{2,\calO}(\bbA)$. There may exist infinitely many $M(\calO)$-conjugacy classes of such generic characters for a specific unipotent orbit $\calO$. 
\begin{example}
	Let $\calO$ be the unipotent orbit corresponding to the partition $(3^21)$ in $\SO_{7}$.  We have
	\[
	V_{2,\calO} = \left\{\begin{pmatrix}
		I_2&X&Y\\
		&I_3&X^\ast\\
		&&I_2
	\end{pmatrix}\in \SO_{7}: X\in \Mat_{2\times 3}, \,X^\ast=-J_3X^TJ_2, \,Y^TJ_2+J_2Y=0\right\}.
	\]
	A choice of generic character is $\psi_\calO: V_{2,\calO}(F)\setminus V_{2,\calO}(\bbA)\rightarrow \bbC^\times$ given by
	\[
	\psi_\calO(v) = \psi(v_{1,3}+v_{2,5}).
	\]
By Pontryagin dualilty, we may identify each generic character with an element in $L_{2,\calO} (F) \cong \Mat_{2\times 3}(F)$. The above character corresponds to the matrix
	\[
	\begin{pmatrix}
		1&0&0\\
		0&0&1
	\end{pmatrix}.
	\]
	The rank of the matrix and the fact that its row space is not totally isotropic are invariant under the action of $M(\calO) \cong \GL_2(F)\times \SO_3(F)$ on $\Mat_{2\times 3}(F)$. Any choice with full rank and non-totally-isotropic row space corresponds to a generic character. 
\end{example}

\subsection{Global Fourier coefficients} \label{gfc}Let $(\pi, \calV)$ be an automorphic representation of $\SO_{2k+1}(\bbA)$. We define the Fourier coefficients of $\pi$ associated with a unipotent orbit $\calO$ by the following:
\begin{definition}\label{Def1}
	Let $\psi_\calO : V_{2,\calO}(F)\setminus V_{2,\calO}(\bbA) \rightarrow \bbC^\times $ be a generic character associated with a unipotent orbit $\calO$ in $\SO_{2k+1}$. For an automorphic function $\varphi\in\pi$,  the Fourier coefficient of $\varphi$ with respect to $\psi_\calO$ is
	\begin{equation}
F_{\psi_\calO}(\varphi_\pi)(g) = \int_{[V_{2,\calO}]}\varphi(ug)\psi_\calO(u) \,du.
	\end{equation}
\end{definition}
	Henceforth, we use $[K]$ to denote $K(F)\setminus K(\bbA)$ for any group $K$. We say that the orbit $\calO$ supports $\pi$ if there exists some $\varphi\in \calV$ and $\psi_\calO$ generic such that the above integral is non-zero. Otherwise, we say that $\calO$ does not support the representation $\pi$.
	
\subsection{Twisted Jacquet modules} \label{sec33} Suppose now $F$ is a non-archimedean local field.  Let $U$ be a unipotent subgroup of $\SO_{2k+1}(F)$, with $\psi_U:U\rightarrow \bbC^\times$ a character on $U$. Let $(\pi,\calV)$ be a smooth representaion of $\SO_{2k+1}(F)$. Suppose there exists a subgroup $M\subset \SO_{2k+1}(F)$ which normalizes $U$ and stablizes the character $\psi_U$.  Consider the subspace $\calV(U,\psi_U)$ of $\calV$ generated by vectors of the form $\{\pi(u)v-\psi_U(u)v\,\vert\, v\in \calV, u\in U\}$. The twisted Jacquet module of $\pi$ with respect to $\psi_U$ is defined by $J_{U,\psi_U}(\pi) =\calV/\calV(U,\psi_U)$. If $\psi_U$ is trivial, we denote it by $J_{U}(\pi)$ and call it the Jacquet module of $\pi$ with respect to $U$. This defines an exact functor between the categories of smooth representations on the two groups 
	\[
	J_{U,\psi_U} : Rep(\SO_{2k+1}(F))\rightarrow Rep(M).
	\]

The Levi subgroup $M(\calO)(F)$  acts on $L_{2,\calO}(F)$, and hence on the character group 
	\[
	\widehat{L_{2,\calO}(F)} \cong L_{2,\calO}(F). 
	\]
	Again, we only look at those generic characters whose stablizer under this action is of the same Cartan type as $M^{\mu_0}(\calO)(\bar{F})$. 
	\begin{definition}
		Let $(\pi,\calV)$ be an admissible representation of $\SO_{2k+1}(F)$. The twisted Jacquet module of $\pi$ associated to a unipotent orbit $\calO$ and a generic character $\psi_\calO: V_{2,\calO}(F)\rightarrow \bbC^\times$ is given by
		\begin{equation}\label{jm}
			J_{V_{2,\calO},\psi_\calO}(\pi).
		\end{equation}
		We say that the unipotent orbit $\calO$ supports $\pi$ if there exists some generic character $\psi_\calO$ such that (\ref{jm}) is non-zero.
	\end{definition}
\subsection{Wave front sets} We have the following definition that applies to both global and local situations.
\begin{definition}
		Let $(\pi,\calV)$ be either an automorphic representation of $\SO_{2k+1}(\bbA)$ (where $\bbA$ is the ring of adeles of a number field $F$) or an admissible representation of $\SO_{2k+1}(F)$ (where $F$ is a non-archimedean local field). The wave front set $\calO(\pi)$ is the set of maximal unipotent orbits in $\SO_{2k+1}$ such that $\calO\in\calO(\pi)$.
	\end{definition}

\section{\textbf{Theta representations and the tower of theta liftings}}\label{sec4}
\subsection{Local theta representations}\label{sec41}
Let $F$ be a non-archimedean local field. In \cite{BumpFG03}, Bump-Friedberg-Ginzburg constructed the local theta representation $\Theta_{2k+1}$ of the double cover $\wt{\SO}_{2k+1}(F)$ as the irreducible image of an intertwining operator 

The theta representation $\Theta_{2k+1}$ of the double cover $\wt{\SO}_{2k+1}(F)$ is a small representation in the terminology of \cite{BumpFG03}. It agrees with the minimal representation when $k=2$ or $3$. When $k= 4$ or $5$, $\calO(\Theta_{2k+1})$ is the singleton set containing the next smallest unipotent orbit. Bump-Friedberg-Ginzburg \cite{BumpFG03} proved the following result:
\begin{proposition}\label{prop62} Let $\Theta_{2k+1}$ be the theta representation of the double cover $\wt{\SO}_{2k+1}(F)$. Let $n$ be a positive integer. Then
\begin{equation}\label{orbit}
	\calO(\Theta_{2k+1})=\begin{cases}
		(2^{2n}1) &\text{if}\; k=2n,\\
		(2^{2n}1^3) &\text{if}\;k =2n+1.
	\end{cases}
\end{equation}
\end{proposition}
	\begin{proof}	
		See either \cite{BumpFG03} or Section 2 of \cite{Ka17}.
	\end{proof}

 Let $r$ be an integer such that $1\leqslant r\leqslant k$. Suppose $P_r= (\GL_r\times \SO_{2k-2r+1})U_r$ is a maximal parabolic subgroup of $\SO_{2k+1}$ with the indicated Levi decomposition, where $U_r$ is the unipotent radical. Recall from \cite{KP84} that the two-fold metapletic cover $\wt{\GL}_r(F)$ of the general linear group $\GL_r(F)$ affords a theta representation $\Theta_{\GL_r}$. Note that the two-fold cover $\wt{\GL}_r(F)$ embeds into $\wt{\SO}_{2k+1}(F)$ as the Levi factor of the inverse image of the parabolic subgroup $P_r(F)$. The reader may refer to \cite{KP84} for more details.  We have the following:
\begin{proposition}\label{prop61}
	Let $\Theta_{2k+1}$ be the local theta representation of $\wt{\SO}_{2k+1}(F)$. Considered as a representation of $\wt{\GL}_r\times\wt{\SO}_{2k-2r+1}(F)$, the Jacquet module  of $\Theta_{2k+1}$ with respect to $U_r$ is ismorphic to $\Theta_{\GL_{r}}\otimes \Theta_{2k-2r+1}$. When $r=k$, $\Theta_1$ is the trivial representation.
\end{proposition}
\begin{proof}
	This is Theorem 2.3 of \cite{BumpFG03}, or Proposition 1 of \cite{BumpFG06}.
\end{proof}

\subsection{Global theta representations}
Suppose now $F$ is a number field, with its ring of adeles  $\bbA$. The global theta representation of $\wt{\SO}_{2k+1}(\bbA)$  is given by the residues of an Eisenstein series constructed in \cite{BumpFG03}. In the terminology of \cite{BumpFG03}, let $\chi_s: T(\bbA) \rightarrow \bbC^\times$ be the character on the split torus attached to the complex parameter $s = (s_1,s_2,\cdots, s_k)\in \bbC^k$. Let $\Ind(\chi_{\theta})$ be the representation of $\wt{\SO}_{2k+1}(\bbA)$ parabolically induced from the character $\chi_\theta:  Z(\wt{T}(F))\setminus Z(\wt{T}(\bbA))\rightarrow \bbC^\times$ attached to the complex parameter $s_\theta=(k/2, (k-1)/2, \cdots, 1/2)$. For any $f\in \Ind(\chi_{\theta})$, there is a unique global smooth section sending any $s\in\bbC^k$ to a function $f_s \in \Ind(\chi_{s})$ such that $f_{s_\theta} = f$. Bump-Friedberg-Ginzburg \cite{BumpFG03} define the Eisenstein series associated to this smooth section by
\begin{equation}
E(g,f_s) = \sum_{\gamma\in B(F)\setminus \SO_{2k+1}(F)}f_s(\gamma g),\quad g\in\wt{\SO}_{2k+1}(\bbA). 
\end{equation}
By taking the residue of this Eisenstein series at $s=s_\theta$, they obtain an automorphic form
\[
\theta_f(g) =\res_{s_1=k/2}\cdots\res_{s_{k-1}=1}\res_{s_k=1/2} E(g,f_s).
\]
Then, the global theta representation $\Theta_{2k+1}$ of $\wt{\SO}_{2k+1}(\bbA)$ is the representation on the subspace of $L^2([\wt{\SO}_{2k+1}])$ spanned by all $\theta_f$ with $f\in \Ind(\chi_{\theta})$.

\begin{proposition}\label{prop1}
	Let $\theta$ be a  function in the theta representation $\Theta_{2k+1}$ of $\wt{\SO}_{2k+1}(\bbA)$. Considered as a function of $(g,h)\in\wt{\GL}_r\times \wt{\SO}_{2k-2r+1}$, the integral
\[
\int_{U_r(F)\setminus U_r(\bbA)} \theta(u(g,h))\,du
\]
is in the space of the automorphic representation $\Theta_{\GL_r}\otimes \Theta_{2k-2r+1}$. Here, $\Theta_{\GL_r}$ is the global exceptional representation on the double cover $\wt{\GL}_{r}(\bbA)$ in the sense of \cite{KP84}.
\end{proposition}

This is the global version of \hyperref[prop61]{Proposition \ref*{prop61}}. The case $r=1$ is proved in \cite{BumpFG03}. Similar statement for the n-fold metaplectic cover $\wt{\Sp}^{(n)}_{2k}$ can be found in \cite{FG16}. The proof is similar to these two sources. For the convenience of the readers, we outline the proof here.
\begin{proof}
	 Consider the exceptional quasicharacter $\chi_\theta: T(\bbA)\rightarrow \bbC^\times$ defined by
	\[
	\chi_\theta(\diag(t_1,t_2,\cdots,t_k,1,t_k,\cdots,t_2,t_1)) =|t_1|^{\frac{k}{2}}|t_2|^{\frac{k-1}{2}}\cdots|t_k|^{\frac{1}{2}}.
	\]
	Denote $\delta_B$ and $\delta_{P_r}$ the modular characters of the standard Borel subgroup of upper triangular matrices and $P_r$ respectively. By induction in stages, we deduce that  $\Ind_{\wt{B}}^{\wt{\SO}_{2k+1}}(\delta_B^{1/2}\chi_\theta)$ is equal to $\Ind_{\wt{P}_r}^{\wt{\SO}_{2k+1}}(\Theta_{\GL_{r}}\otimes \Theta_{2k-2r+1})\delta_{P_r}^{\frac{3n-2r+1}{4n-2r}}$.
	
	 Let $E_{P_r}(g,s, f_s)$ be the Eisenstein series of $\wt{\SO}_{2k+1}(\bbA)$ associated to 
	 \[
	 f_s\in \Ind_{\wt{P}_r}^{\wt{\SO}_{2k+1}}(\Theta_{\GL_{r}}\otimes \Theta_{2k-2r+1})\delta_{P_r}^{s}.
	 \] It follows from \cite{BumpFG03} that any function in the representation $\Theta_{2k+1}$ is the residue of $E_{P_r}(g,s, f_s)$ at the point $s=\frac{3n-2r+1}{4n-2r}$. Consider the constant term
	\begin{equation}\label{p1:1}
	\int_{U_r(F)\setminus U_r(\bbA)} E_{P_r}(u(g,h),s, f_s)\,du.
	\end{equation}
	For $\Re(s)$ large enough we can unfold the Eisenstein series and obtain that (\ref{p1:1}) equals
	\begin{equation}
	\sum_{\omega\in P_r(F)\setminus \SO_{2k+1}(F)/P_r(F)}\int_{U_r^\omega(F)\setminus U_r(\bbA)}\sum_{\gamma\in (\omega^{-1}P_r(F)\omega \cap P_r(F))\setminus P_r(F)}f_s(\omega^{-1} \gamma u(g,h))\,du,
	\end{equation}
	where $U_r^\omega = \omega^{-1}U_r\omega \cap \omega_0^{-1}U_r\omega_0$ with $\omega_0$ being the long Weyl group element. 
	
	The double cosets  $P_r(F)\setminus \SO_{2k+1}(F)/P_r(F)$ may be represented by Weyl group elements. By Section II.1.7 of \cite{MW95}, any of the inner summations is either an Eisenstein series or a product of such series. All such Eisenstein series are holomorphic at $s= \frac{3n-2r+1}{4n-2r}$ except the one corresponding to the long Weyl group element. Following the same argument as in Proposition 1 of \cite{FG16}, by taking the residue of (\ref{p1:1}) at $s=\frac{3n-2r+1}{4n-2r}$, we obtain
	\begin{equation}
\int_{U_r(F)\setminus U_r(\bbA)} \theta(u(g,h))\,du=\Res_{s=\frac{3n-2r+1}{4n-2r}}M_{\omega_0}f_s(g,h),
	\end{equation}
	where
\begin{equation}
	M_{\omega_0}f_s(g,h) = \int_{\omega_0^{-1}U_r(F)\omega_0\setminus U_r(\bbA)}f_s(\omega_0^{-1}u(g,h))\,du.
\end{equation}
As $\Res_{s=\frac{3n-2r+1}{4n-2r}}M_{\omega_0}f_s(g,h)$ is an element of $\Theta_{\GL_r}\otimes \Theta_{2k-2r+1}$, the proof is complete.
	\end{proof}
The global theta representation has the same smallness property given by \ref{orbit}. We have the following statement:
\begin{proposition}\label{prop2} Let $\Theta_{2k+1}$ be the global theta representation of the double cover $\wt{\SO}_{2k+1}(\bbA)$. Let $n$ be a positive integer. Then
	\begin{equation}\label{orbit}
		\calO(\Theta_{2k+1})=\begin{cases}
			(2^{2n}1) &\text{if}\; k=2n,\\
			(2^{2n}1^3) &\text{if}\;k =2n+1.
		\end{cases}
	\end{equation}
\end{proposition}
\begin{proof}
	See Theorem 4.2(i) of \cite{BumpFG03}, or Proposition 2 of \cite{BumpFG06}.
	\end{proof}
There is another important property we need for the theta representations $\Theta_{2k+1}$. If $r<k/4$ is a positive integer, we denote
 $H_{r,2k+1}$ the unipotent radical of the maximal parabolic subgroup of $\SO_{2k+1}$ whose Levi part is $\GL_r\times \SO_{2k-2r+1}$. In other words, $H_{r,2k+1}$ consists of upper triangular matrices of the form
\[
\left\{\begin{pmatrix}
I_r&x&\ast\\
&I_{2k-2r+1}&x^\ast\\
&&I_r
\end{pmatrix}\in \SO_{2k+1}: x^\ast = -J_{2k-2r+1}x^TJ_r \right\},
\]
where $\ast$ denotes whatever is needed for the matrix to be orthogonal. Similarly, let $H_{r,2k-2r+1}$ be the unipotent of the standard maximal parabolic subgroup of $\SO_{2k-2r+1}$ with Levi subgroup $\GL_{r}\times\SO_{2k-4r+1}$. Via the embedding (\ref{ebd}), we identify $H_{r,2k-2r+1}$ with its image in $\SO_{2k+1}$. Hence, $H_{r,2k-2r+1}$ consists of matrices of the form
\[\left\{
\begin{pmatrix}
	I_r&&&&\\
	&I_r&y&\ast&\\
	&&I_{2k-4r+1}&y^\ast&\\
	&&&I_r&\\
	&&&&I_r
\end{pmatrix}\in \SO_{2k+1}(\bbA) \right\}.
\]

For any $u = (u_{i,j})\in H_{r,2k+1}(\bbA)$, we define the character $\psi_1: H_{r,2k+1}(F)\setminus H_{r,2k+1}(\bbA)\rightarrow \bbC^\times $ by
\[
\psi_1(u) =\psi(\sum_{j=1}^r u_{j,j+r}).
\]
\begin{proposition}\label{prop3}
Fix a function $\theta \in \Theta_{2k+1}$.  The integral
		\begin{equation}\label{peq:01}
		f(g) =\int_{ H_{r,2k+1}(F)\setminus H_{r,2k+1}(\bbA)}\theta(ug)\psi_1(u) \,du
		\end{equation}
	is left invariant by $H_{r,2k-2r+1}(\bbA)$. That is, $f(g) = f(vg)$ for any $ v\in H_{r,2k-2r+1}(\bbA)$.
\end{proposition}
\begin{proof}
      Via the embedding (\ref{ebd}), the center $Z(H_{r,2k-2r+1})$ consists of matrices of the form
	\[\left\{
	\begin{pmatrix}
	I_r&&&&\\
	&I_r&&z&\\
	&&I_{2k-4r+1}&&\\
	&&&I_r&\\
	&&&&I_r
	\end{pmatrix}:\quad z\in \Mat_{r\times r},\quad  z^TJ_r+J_rz=0\right\}.
	\]
	We first expand (\ref{peq:01}) against $Z(R_{r,2k-2r+1})(F)\setminus Z(R_{r,2k-2r+1})(\bbA)$.  Embed the group $\GL_r(F)$  into $\SO_{2k+1}(F)$ via
	\[
	h\hookrightarrow \diag(I_r,h,I_{2k-4r+1},h^\ast, I_r), \quad \forall h\in\GL_r(F).
	\]
	The $\GL_r(F)$-action on $Z(H_{r,2k-2r+1})(F)\setminus Z(H_{r,2k-2r+1})(\bbA)$ induces an action on its character group, which may be identified with $Z(H_{r,2k-2r+1})(F)$. This action preserves the rank of the matrices in $Z(H_{r,2k-2r+1})(F)\cong \left\{z\in \Mat_{r\times r},\, z^TJ_r+J_rz=0\right\}$. On the other hand, the rank of any $z\in Z(H_{r,2k-2r+1})(F)$ must be even. Suppose $\Rank(z) = 2q$, where $0<2q\leqslant r$. We may choose a representative 
	\begin{equation}\label{zr}
	\begin{pmatrix}
	0&z_q\\
	0&0
	\end{pmatrix}\in\Mat_{r\times r}(F),\end{equation}
where
\[
z_q=\diag(\lambda_1,\lambda_2,\cdots,\lambda_q,-\lambda_q,\cdots,-\lambda_2,-\lambda_1)\in \Mat_{2q\times 2q}(F), \lambda_i \in F^\times \quad \forall i =1,2,\cdots,q.
\]
This gives the corresponding character on $[Z(H_{r,2k-2r+1})]$ by:
	\[
	\psi_{Z,q}(v) = \psi(\lambda_1v_{1,r-2q+1}+\lambda_2v_{2,r-2q+2}+\cdots +\lambda_qv_{q,r-q}),
	\]
	where 
	\[
	v\in Z(H_{s,2k-2s+1})(\bbA)\cong\Mat_{r\times r}(\bbA).
	\]
  For any character in the same orbit as $\psi_{Z,q} $, its contribution in the expansion is given by the integral
  	\begin{equation}\label{peq:02}
  \int_{ [H_{r,2k+1}]}\int_{[Z(H_{r,2k-2r+1})]} \theta(uvg)\psi_1(u) \psi_{Z,q}(v)\,dvdu.
  \end{equation}
 Notice that the character $\psi_1\psi_{Z,q}$ coincides with a generic character associated to the unipotent orbit $\calO_q=(4^{2q}3^{r-2q}1^{2k-2q-3r+1})$ (or $(4^r1^{2k-4r+1})$ in the case $r=2q$). Moreover, we have $V_{2,\calO_q} \subseteq  H_{r,2k+1}Z(H_{r,2k-2r+1})$. As a result, the integral (\ref{peq:02}) contains an inner integral that is a Fourier coefficient of $\theta$ with respect to the unipotent orbit $\calO_q$. The integral (\ref{peq:02}) is then zero because any such Fourier coefficient of $\theta$ is zero by \hyperref[prop2]{Proposition \ref*{prop2}}. Thus, only the integral corresponding to the trivial character contributes, and (\ref*{peq:01}) is equal to
 	\begin{equation}\label{peq:03}
 f(g) =\int_{[H_{r,2k+1}Z(H_{r,2k-2r+1})]}\theta(ug)\psi_1(u) \,du.
 \end{equation}
 
 As the center $Z(H_{r,2k-2r+1})$ is now contained in the domain of integration, we can further expand (\ref{peq:03}) against \[[H_{r,2k-2r+1}/Z(H_{r,2k-2r+1})]\cong \Mat_{r\times (2k-4r+1)}(F)\setminus \Mat_{r\times (2k-4r+1)}(\bbA).\]
 Embed the group $\GL_r(F)\times \SO_{2k-4r+1}(F)$ into $\SO_{2k+1}(F)$ via
 \[
 (h_1,h_2)\hookrightarrow \diag(I_r,h_1,h_2,h_1^\ast,I_r),\quad (h_1,h_2)\in \GL_r(F) \times \SO_{2k-4r+1}(F).
 \]
 The conjugation action of $\GL_r(F)\times \SO_{2k-4r+1}(F)$ on the quotient $[H_{r,2k-2r+1}/Z(H_{r,2k-2r+1})]$ induces an action on the character group of the latter, which may be  identified with $\Mat_{r\times (2k-4r+1)}(F)$. For any $\xi\in\Mat_{r\times (2k-4r+1)}(F)$ with $\Rank(\xi) =q,\, 1\leqslant q\leqslant r$, if any of its row vectors is non-isotropic with respect to the bilinear form on $F^{2k-4r+1}$ given by
 \[
 (x,y)\mapsto xJ_{2k-4r+1}y^T,
 \]
 then the corresponding contribution is given by the integral 
  \begin{equation}\label{eq11}
 	\int_{ [H_{r,2k+1}]}\int_{[H_{r,2k-2r+1}]} \theta(uvg)\psi_1(u) \psi_{\xi}(v)\,dvdu.
 \end{equation}
 The product of the characters $\psi_1$ and $\psi_{\xi}$ is a generic character associated to the  unipotent orbit $(5^{q}3^{r-q}1^{2k-3r-2q+1})$. Thus, the integral (\ref{eq11}) contains a Fourier coefficient of $\Theta_{2k+1}$ associated to the unipotent orbit $(5^{q}3^{r-q}1^{2k-3r-2q+1})$, which is zero by \hyperref[prop2]{Proposition \ref*{prop2}}. Under the conjugation action by $\GL_{r}(F)\times \SO_{2k-4r+1}(F)$, any $\xi\in \Mat_{r\times (2k-4r+1)}(F)$ whose row space is not totally isotropic lies in the same orbit as the non-isotropic ones with the same rank. The same argument shows that the corresponding contribution is zero. Therefore, the only possible non-zero contributions are from those $\xi\in \Mat_{r\times (2k-4r+1)}(F)$ whose row space is totally isotropic. We may pick the representatives of these orbits to be
 \[
 z_q=\begin{pmatrix}
 I_q&0\\
 0&0
 \end{pmatrix} \in \Mat_{r\times (2k-4r+1)}(F), q=0,1,\cdots, r.
 \]
 
 For any $z_q\in \Mat_{r\times (2k-4r+1)}(F)$, the correponding contribution in the expansion is 
 \begin{equation}\label{peq:04}
 \int_{ [H_{r,2k+1}]}\int_{[H_{r,2k-2r+1}]} \theta(uvzg)\psi_1(u) \psi_{2,q}(v)\,dvdu,
 \end{equation}
 where the character $\psi_{2,q}$ corresponding to $z_q$  is the trivial character when $q=0$ and is given by
 \[
 \psi_{2,q}(u) =\psi(u_{1,1}+u_{2,2}+\cdots+u_{q,q}),\quad \forall u\in  H_{r,2k-2r+1}/Z(H_{r,2k-2r+1})(\bbA) \cong \Mat_{r\times (2k-4r+1)}(\bbA)
 \]
 when $q>0$.  We claim that the only non-zero contribution is when $q=0$. Suppose on the contrary that $q>0$. We can further expand (\ref{peq:04}) against $[Z(H_{r,2k-4r+1})]$ and follow by another expansion against the abelian quotient $[H_{r,2k-4r+1}/Z(H_{r,2k-4r+1})]$. Here, $H_{r,2k-4r+1}$ is similarly defined as the unipotent subgroup of $\SO_{2k+1}$ consisting of matrices of the form
	\[\left\{
\begin{pmatrix}
I_{2r}&&&&\\
&I_r&y&\ast&\\
&&I_{2k-6r+1}&y^\ast&\\
&&&I_r&\\
&&&&I_{2r}
\end{pmatrix},\quad y\in \Mat_{r\times (2k-6r+1)}\right\}.
\]
For any character $\psi^\ast$ on $[H_{r,2k-4r+1}/Z(H_{r,2k-4r+1})]$, the corresponding contribution is 
 \begin{equation}\label{peq:05}
\int_{ [H_{r,2k+1}]}\int_{[H_{r,2k-2r+1}]} \int_{[H_{r,2k-4r+1}]}\theta(uvwg)\psi_1(u) \psi_{2,q}(v)\psi^\ast (w)\,dwdvdu.
\end{equation}
Any character $\psi^\ast$ corresponding to a matrix in $\Mat_{r\times (2k-6r+1)}(F)$ that contains non-isotropic row vectors in $F^{2k-6r+1}$ again contributes zero. 

We claim that the constant term is also zero. If we denote 
\[H_r=H_{r,2k+1}H_{r,2k-2r+1}H_{r,2k-4r+1},\]
then the constant term is of the form 
 \begin{equation}\label{peq:06}
\int_{ [H_r]} \theta(ug)\psi_{1,q}(u)\,du.
\end{equation}
Here, if $u=u_1u_2u_3$ with $u_1\in H_{r,2k+1}(\bbA), u_2\in H_{r,2k-2r+1}(\bbA)$ and $u_3\in H_{r,2k-4r+1}(\bbA)$, then
\[
\psi_{1,q}(u) =\psi_1(u_1)\psi_{2,q}(u_2).
\]
This integral is a Fourier coefficient of the constant term of $\Theta_{2k+1}$ with respect to the parabolic subgroup whose Levi part is $\GL_{3r}\times \SO_{2k-6r+1}$. By \hyperref[prop1]{Proposition \ref*{prop1}}, we may regard this integral as a function in the representation $\Theta_{\GL_{3r}}\otimes \Theta_{2k-6r+1}$ of  $\wt{\GL}_{3r}(F)\times \wt{\SO}_{2k-6r+1}(F)$. According to  \cite{CAI19}, the corresponding Fourier coefficient of $\Theta_{\GL_{3r}}$ is the semi-Whittaker coefficient associated with the partition $\Lambda =(3^q2^{r-q}1^{r-q})$. Following \cite{CAI19}, let $P_\Lambda =M_\Lambda U_\Lambda$ be the standard parabolic subgroup of $\GL_{3r}$ with the Levi subgroup $M_\Lambda \cong \GL^{q}_3\times \GL_2^{r-q}\times \GL_1^{r-qr}$ and $U_\Lambda$ the unipotent radical. Let $U$ be the unipotent radical of the standard Borel subgroup of $\GL_{3r}$ and $\psi_\Lambda : U(F)\setminus U(\bbA)\rightarrow \bbC^\times$ be the character such that it acts non-trivially as $\psi$ on the one-parameter subgroup corresponding to simple positive root contained in $M_\Lambda$ and act trivially otherwise. Then  the $\Lambda$-semi-Whittaker coefficient of any function $\theta_{\GL_{3r}}$ in $\Theta_{\GL_{3r}}$ is given by
\begin{equation}
\int_{U(F)\setminus U(\bbA)}\theta_{\GL_{3r}}(ug)\psi_\Lambda(u)\,du.
\end{equation}
 By Proposition 4.1 of \cite{CAI19}, any such semi-Whittaker coefficient is zero. 
 
 Therefore, we only need to consider those terms corresponding to characters on $[H_{r,2k-4r+1}]$ represented by 
\begin{equation}\label{eq12}
\psi_{3,q'}(u) = \psi(y_{1,1}+y_{2,2}+\cdots+y_{q',q'}),
\end{equation}
with $1<q'\leqslant r$ and
\begin{equation}
u=\begin{pmatrix}
	I_{2r}&&&&\\
	&I_r&y&\ast&\\
	&&I_{2k-6r+1}&y^\ast&\\
	&&&I_r&\\
	&&&&I_{2r}
\end{pmatrix}\in H_{r,2k-4r+1}(\bbA), y =(y_{i,j}) \in \Mat_{r\times (2k-6r+1)}(\bbA).
\end{equation}

We continue this argument by further expanding the integral (\ref{eq12}) against relevant unipotent subgroups. For each step, we either obtain a Fourier coefficient of $\Theta_{2k+1}$ associated to unipotent orbit that is incomparable to $\calO(\Theta_{2k+1})$, or we obtain some semi-Whittaker coefficient  on the double cover $\wt{\GL}_{nr}$ corresponding to a partition of the integer $nr$ (with $n>3$) that contains an integer greater than $2$. The former is zero by \hyperref[prop2]{Proposition \ref*{prop2}}, while the latter is also zero by Proposition 4.1 of \cite{CAI19}. This shows that the only contribution of the expansion of (\ref{peq:03}) against  $[H_{r,2k-2r+1}/Z(H_{r,2k-2r+1})]$ is the constant term, which completes the proof.
 \end{proof}
For intergers $r_1\leqslant r_2\leqslant\cdots\leqslant r_n$, we further denote $H(r_1,r_2,\cdots,r_{n-1};r_n)$ the unipotent radical of the parabolic subgroup of $\SO_{2k+1}$ with Levi subgroup
\[L(r_1,r_2,\cdots,r_{n-1};r_n)=\GL^2_{r_1}\times \GL^2_{r_2}\times \cdots\times \GL^2_{r_{n-1}}\times\GL_{r_n}\times \SO_{2k-4(r_1+\cdots+r_{n-1})-2r_n+1}.\]
We also let
\[
L=\GL_{2r_1}\times \GL_{2r_2}\times \cdots\times \GL_{2r_{n-1}}\times\SO_{2k-4(r_1+\cdots+r_{n-1})+1}.
\]
Denote the diagonal embedding $\iota^\ast : L(r_1,r_2,\cdots,r_{n-1};r_n) \subset L \hookrightarrow \SO_{2k+1}$.

For each $i=1,2,\cdots, n-1$, let $H_{r_i}$ be the unipotent radical of the standard Siegel parabolic subgroup of  $\GL_{2r_i}$ whose Levi part is $\GL_{r_i}\times \GL_{r_i}$. Also, let  $H_{r_n}$ be the unipotent radical of the maximal parabolic subgroup of $\SO_{2k-4(r_1+\cdots+r_{n-1})+1}$ with Levi subgroup $\GL_{r_n}\times \SO_{2k-4(r_1+\cdots+r_{n-1})-2r_n+1}$. Then we define 
\[
H^0 :=H_{r_1}\times H_{r_2}\times \cdots\times H_{r_{n-1}}\times H_{r_n}\subseteq H(r_1,r_2,\cdots,r_{n-1}; r_n).
\]

We will define a character on $[H^0]$ and then extend it trivially to $[H(r_1,r_2,\cdots,r_{n-1};r_n)]$. For any $u_i\in H_{r_i}(\bbA)$ with $i=1,2,\cdots, n-1$, we may write 
\[
u_i =\ \begin{pmatrix}
I_{r_i}&x_i\\
&I_{r_i}
\end{pmatrix}, \quad x_i\in \Mat_{r_i\times r_i}(\bbA).
\]
Define $\psi_{r_i}: H_{r_i}(F)\setminus H_{r_i}(\bbA)\rightarrow \bbC^\times$ by
\[
\psi_{r_i} (u_i) = \psi(\Tr(x_i)).
\]
For the factor $H_{r_n}$, define the character by
\[
\psi_{r_n}(u) =\psi(\sum_{j=1}^{r_n} u_{j,j+r_n}), \quad u\in H_{r_n}(\bbA).
\]
Pulling back each $\psi_{r_i}$ via the projection of $H^0$ onto the corresponding factor and taking the product afterwards, we obtain $\psi_n := \prod_{i=1}^n\psi_{r_i}: [H^0]\rightarrow \bbC^\times$. Extend $\psi_n$ trivially to $[H(r_1,r_2,\cdots,r_{n-1};r_n)]$. Furthermore, let $H_{n}$ be the unipotent radical of the standard maximal parabolic subgroup of $\SO_{2k-4(r_1+\cdots+r_{n-1})-2r_n+1}$ with Levi subgroup given by  
\[
\GL_{r_n}\times \SO_{2k-4(r_1+\cdots+r_{n-1})-2r_n+1}.
\] 
Embed $H_{n}$ into $\SO_{2k+1}$ via (\ref{ebd}) and still denote its image by $H_n$. \hyperref[prop3]{Proposition \ref*{prop3}} admits a straightforward corollary.
\begin{corollary}\label{cor}
	The function
	\begin{equation}\label{peq:07}
	f(g) =\int_{ [H(r_1,r_2,\cdots,r_{n-1};r_n)]}\theta(ug)\psi_n(u) \,du
	\end{equation}
	
	is left invariant under $H_{n}(\bbA)$.
\end{corollary}
\begin{proof}
	Apply \hyperref[prop1]{Proposition \ref*{prop1}} with $r=2(r_1+r_2+\cdots+r_{n-1})$. Then apply \hyperref[prop3]{Proposition \ref*{prop3}} on the theta representation of the smaller orthogonal group.
	\end{proof}
\subsection{The tower of the theta liftings}
Let $(\pi,\calV)$  be an irreducible cuspidal genuine automorphic representation of $\wt{\SO}_{2k+1}(\bbA)$. Suppose $\SO_{2k'}$ is a split even orthogonal group. By identifying $\SO_{2k'}\times \SO_{2k+1}$ with its embedded image in $\SO_{2k+2k'+1}$ via (\ref{ebd}), we consider functions on $\wt{\SO}_{2k'}(\bbA)$ of the form
\begin{equation}\label{lift}
f(h) = \int_{[\SO_{2k+1}]} \varphi(g)\bar{\theta}_{2k+2k'+1}(h,g)\,dg,
\end{equation}
where $\varphi$ is any vector in $\calV$ and $\theta_{2k+2k'+1}$ is any function in the representation space of $\Theta_{2k+2k'+1}$. This integral defines a map from the irreducible cuspidal genuine automorphic representations on $\wt{\SO}_{2k+1}(\bbA)$ to a genuine automorphic representation $\Theta_{2k+2k'+1}(\pi)$ on $\wt{\SO}_{2k'}(\bbA)$.

By fixing the representation $\pi$ of $\wt{\SO}_{2k+1}(\bbA)$ and varying the theta representations $\Theta_{2k+2k'+1}$ of $\wt{\SO}_{2k+2k'+1}(\bbA)$ with increasing $k'$, we obtain a tower of liftings of representations on  $\wt{\SO}_{2k'}(\bbA)$:
\begin{center}
	\begin{tikzcd}

	& &\Theta_{2k+2k'+3}(\pi) \\
	&& \Theta_{2k+2k'+1}(\pi)\\
	& &\vdots\\
	\pi \arrow[rr,"\vdots"]\arrow[uurr,end anchor={[xshift=-4ex]}]\arrow[uuurr,end anchor={[xshift=-4ex]}]
	& &\Theta_{2k+3}(\pi).	\end{tikzcd}
\end{center}
In \cite{BumpFG06}, Bump-Friedberg-Ginzburg show that if $\Theta_{2k+2k'+1}(\pi) =0$, then $\Theta_{2k+2k'-1}(\pi) =0$. It is also proved in \cite{BumpFG06} that any genuine cuspidal automorphic representation $\pi$ of $\wt{\SO}_{2k+1}(\bbA)$ lifts nontrivially to an automorphic representation on $\wt{\SO}_{8k}(\bbA)$. This raises the question of when the first non-zero theta lifting occurs along the tower for a fixed $\pi$. In the case when $\pi$ is generic, a conjecture in \cite{BumpFG06} states that $\pi$ should lift non-trivially to an automorphic representation of $\wt{\SO}_{2k+4}(\bbA)$. Those authors also proved the following result:
\begin{theorem}
	Let $\pi$ be an irreducible cuspidal genuine automorphic representation of $\wt{\SO}_{2k+1}(\bbA)$. If the representation $\Theta_{4k+5}(\pi)$ of $\wt{\SO}_{2k+4}(\bbA)$  is generic, then the representation $\pi$ is also generic.
\end{theorem}

On the other hand, there is not much known yet for the theta liftings when $\pi$ is non-generic. Motivated by the generic case, we make the following more general conjecture:
\begin{conjecture}\label{Con1}
	Let $(\pi,\calV)$ be an irreducible cuspidal genuine automorphic representation of $\wt{\SO}_{2k+1}(\bbA)$. Suppose 
	\[
	\calO = \left((2n_1+1)^{r_1}(2n_2+1)^{r_2}\cdots(2n_p+1)^{r_p}\right) \in\calO(\pi) 
	\]
	with $n_1>n_2>\cdots >n_p\geqslant 0$ and $r_i>0$ for all $i$. Then $\pi$ lifts nontrivially to an automorphic representation $\Theta(\pi)$ of $\wt{\SO}_{2k+2l+2}(\bbA)$ such that
	\[
	\calO' =  \left((2n_1+3)^{r_1}(2n_2+3)^{r_2}\cdots(2n_p+3)^{r_p}(1)\right)\in\calO\left(\Theta(\pi)\right).\\
	\]
\end{conjecture}

Recall that $l=r_1+r_2+\cdots +r_p$ is the length of the partition corresponding to $\calO$. If $\pi$ is an irreducible cuspidal generic automorphic representation of $\wt{\SO}_{2k+1}(\bbA)$, then $\calO(\pi) =(2k+1)$. \hyperref[Con1]{Conjecture \ref*{Con1}} predicts that it lifts to an automorphic representation $\Theta(\pi)$ on $\wt{\SO}_{2k+4}(\bbA)$ with $\calO(\Theta(\pi)) = ((2k+3)(1))$, which agrees with the conjecture proposed in \cite{BumpFG06}.

For the remaining of this section, assume that the wave front set $\calO(\rho)$ is a singleton set for an automorphic representation $\rho$. Recall that the Gelfand-Kirillov dimension of the representation $\rho$ (see \cite{FG19}) is given by 
\[
\dim(\rho) = \frac{1}{2}\dim(\calO(\rho)) =\dim(V_{2,\calO(\rho)}) +\frac{1}{2}\dim (V_{1,\calO(\rho)}/V_{2,\calO(\rho)}).
\]
\begin{proposition} \label{deqp}
	Suppose $(\pi,\calV)$ is an irreducible cuspidal genuine automorphic representation of $\wt{\SO}_{2k+1}(\bbA)$ with $\calO(\pi)=\calO $ such that its theta lifting  $\Theta(\pi)$ on $\wt{\SO}_{2k+2l+2}(\bbA)$ has $ \calO\left(\Theta(\pi)\right)=\calO'$. Then
\begin{equation}\label{deq}
\dim(\SO_{2k+1}) +\dim (\Theta(\pi))= \dim (\pi)+\dim (\Theta_{4k+2l+3}).
\end{equation}
\end{proposition}

Before we verify (\ref{deq}), we need to fix some notations. Suppose, as in \hyperref[Con1]{Conjecture \ref*{Con1}},  
\[\calO =  \left((2n_1+1)^{r_1}(2n_2+1)^{r_2}\cdots(2n_p+1)^{r_p}\right),\quad n_1>n_2>\cdots >n_p\geqslant 0.
\]
Denote
\begin{equation}
	s_i:=\sum_{j=1}^{i}r_j,\quad i=1,2,\cdots, p.
\end{equation}
In particular, $l= s_p$ is the length of the partition corresponding to $\calO$. The fact that $\calO$ is odd implies that $V_{1,\calO} = V_{2,\calO}$, and both of these are the unipotent radical of the parabolic subgroup $P_{\calO}$ whose Levi part is  
\[
M(\calO) = \GL^{n_1-n_2}_{s_1}\times \GL^{n_2-n_3}_{s_2}\times \cdots\times \GL^{n_{p-1}-n_p}_{s_{p-1}}\times\GL_l^{n_p}\times  \SO_{l}.\]
For simplicity, we denote $U_\calO=V_{1,\calO} = V_{2,\calO}$.

Likewise, we denote by $V_{\calO'}$ the unipotent subgroup $V_{1,\calO'} = V_{2,\calO'}$ associated to the unipotent orbit
\[
\calO'= \left((2n_1+3)^{r_1}(2n_2+3)^{r_2}\cdots(2n_p+3)^{r_p}(1)\right)
\]  
in $\SO_{2k+2l+2}$. The corresponding Levi subgroup is 
\[
M(\calO') = \GL^{n_1-n_2}_{s_1}\times \GL^{n_2-n_3}_{s_2}\times \cdots\times \GL^{n_{p-1}-n_p}_{s_{p-1}}\times \GL^{n_p+1}_l\times \SO_{l+1}.
\]
\begin{proof}
Observe that  $\calO(\Theta_{4k+2l+3}) = (2^{2k+l+1}1)$ since $l$ is odd, and $\dim (\Theta_{4k+2l+3})= \frac{(2k+l+1)^2}{2}$. Also, $\dim(\SO_{2k+1}) = 2k^2+k$. 
Although the dimension of the representations $\Theta_{4k+2l+3}(\pi)$ and  $\pi $ may vary, it suffices to check that the difference between the dimensions agrees with that between $\dim (\Theta_{4k+2l+3})$ and $\dim(\SO_{2k+1})$. 

Notice that the difference between $\dim (\pi)$ and the dimension of the unipotent radical of the Borel subgroup of $\SO_{2k+1}$ is related to that between  $\dim (\Theta_{4k+2l+3}(\pi))$ and the dimension of the unipotent radical of the Borel subgroup of $\SO_{2k+2l+2}$. The former is precisely the dimension of the unipotent radical of the Borel subgroup of $M(\calO)$. We denote this dimension by $t+\frac{(l-1)^2}{4}$, where $\frac{(l-1)^2}{4}$ and $t$ are the dimensions of the maximal unipotent subgroup of the factor $\SO_l$ and the remaining Levi factors respectively. Hence, we obtain that $\dim(\pi) = k^2-t-\frac{(l-1)^2}{4}$. Similarly, we have $\dim(\Theta_{4k+2l+3}(\pi)) = (k+l+1)(k+l) -t-\frac{l^2-1}{4}-\frac{(l-1)(l)}{2}$. Therefore, the difference is exactly 
\begin{equation}
\frac{l^2+1}{2} +2kl+k+l = \frac{(2k+l+1)^2}{2}-(2k^2+k) = \dim(\Theta_{4k+2l+3})-\dim(\SO_{2k+1}).
\end{equation}
\end{proof}
\hyperref[deqp]{Proposition \ref*{deqp}} shows that \hyperref[Con1]{Conjecture \ref*{Con1}} agrees with the dimension equation proposed in \cite{GIN06}. The general philosophy of the dimension equation is that the sum of the dimensions of the representations is equal to the sum of the dimensions of the groups in the domain of integration in a global unipotent integral. In our case, this is given by equation (\ref{deq}). Refer to \cite{GIN06}, \cite{GIN14} and \cite{FG19} for more details on dimension equations.
\section{\textbf{Global theory}}\label{sec5}
We follow the notations in and after \hyperref[Con1]{Conjecture \ref*{Con1}}. Suppose $(\pi,\calV)$ is an irreducible genuine cuspidal automorphic representation of $\wt{\SO}_{2k+1}(\bbA)$. Let $\calO$ be the unipotent orbit of $\SO_{2k+1}$ such that
\[
\calO =\left((2n_1+1)^{r_1}(2n_2+1)^{r_2}\cdots(2n_p+1)^{r_p}\right), \quad n_1>n_2>\cdots >n_p\geqslant 0, \quad r_i>0\quad\forall i=1,2,\cdots,p.
\]
Let $l=r_1+r_2+\cdots +r_p$ be the length of the partition corresponding to $\calO$. Suppose $ \Theta_{4k+2l+3}(\pi)$ is the automorphic representation of $\wt{\SO}_{2k+2l+2}(\bbA)$ obtained by the theta lifting from $\pi$ via integrating functions in $\calV$ against the theta integral kernel $\Theta_{4k+2l+3}$ in the form of (\ref{lift}). As we are only concerned about a fixed theta lifting in this section, we suppress the subscript and simply let $\Theta(\pi) =  \Theta_{4k+2l+3}(\pi)$. We denote by $\calO'$ the unipotent orbit of the group $\SO_{2k+2l+2}$ associated to the partition
\[
\calO' =\left((2n_1+3)^{r_1}(2n_2+3)^{r_2}\cdots(2n_p+3)^{r_p}(1)\right)
\]
Recall that $V_{\calO'} =V_{1,\calO'} = V_{2,\calO'}$ is the unipotent radical of the parabolic subgroup of $\SO_{2k+2l+2}$ whose Levi subgroup is 
\[
M(\calO') = \GL^{n_1-n_2}_{s_1}\times \GL^{n_2-n_3}_{s_2}\times \cdots\times \GL^{n_{p-1}-n_p}_{s_{p-1}}\times \GL^{n_p+1}_l\times \SO_{l+1}.
\]
Notice that
\begin{equation}\label{501}
V_{\calO'}/V_{\calO'}^{(1)} \cong \Bigg(\bigoplus_{j=1}^{p-1} \Mat_{s_j\times s_{j+1}}\Bigg)\oplus \Bigg(\bigoplus_{j=1}^{p-1} \Mat_{s_j\times s_j}^{n_j-n_{j+1}-1}\Bigg)\oplus \Mat_{l\times l}^{n_p}\oplus \Mat_{l\times (l+1)}.
\end{equation}

We first define a character on $[V_{\calO'}/V_{\calO'}^{(1)}] $ by specifying it on each of the components, and then extend it trivially to a character on $[V_{\calO'}]$. For any abelian group $\Mat_{i\times j}$, we may identify the character group of $[\Mat_{i\times j}]$ with $\Mat_{i \times j}(F)$ via
\begin{equation}\label{chr}
X\in\Mat_{i\times j}(F)\mapsto \psi_X: Y\mapsto \psi(XY^T).
\end{equation}
Recall 
\begin{equation}
	s_i:=\sum_{j=1}^{i}r_j,\quad i=1,2,\cdots, p.
\end{equation}
Let
\[
I_{s_j,s_{j+1}} =\begin{pmatrix}
I_{s_j}&0
\end{pmatrix}\in \Mat_{s_j\times s_{j+1}}(F),\quad j= 1,2,\cdots ,p-1,
\]
and 
\[
I_{l,a} = \begin{pmatrix}
I_{\frac{l-1}{2}}&&&\\
&\frac{1}{2}&-a&\\
&&&-I_{\frac{l-1}{2}}
\end{pmatrix}\in\Mat_{l\times (l+1)}(F), \quad a\in F^\times.
\]
Consider the following characters each defined on the respective component in (\ref{501}):
\begin{equation}\label{charv}\begin{cases}
\psi_{s_j\times s_j}(v) =\psi(\Tr v)  &\text{if}\;v\in \Mat_{s_j\times s_j}(\bbA), j=1,2,\cdots, p, \\
\psi_{s_j\times s_{j+1}}(v) =\psi(\Tr (I_{s_j,s_{j+1}} v^T))& \text{if}\; v\in \Mat_{s_j\times s_{j+1}}(\bbA), j=1,2,\cdots, p-1,\\
\psi_{l\times (l+1)}(v) =\psi(\Tr (I_{l,a} v^T))&\text{if}\;v\in  \Mat_{l\times (l+1)}(\bbA).
\end{cases}
\end{equation}
Pulling back each of these characters via the projection map onto the respective component and taking the product afterwards, we obtain a character on $[V_{\calO'}/V_{\calO'}^{(1)}] $. Extend it to a character on $[V_{\calO'}]$ and denote the resulting character by $\psi_{a,V_{\calO'}}$.  This is a generic character attached to the unipotent orbit $\calO'$.

For any automorphic function in $\Theta(\pi)$ of the form
\[
f(h) = \int_{\SO_{2k+1}(F)\setminus\SO_{2k+1}(\bbA)} \varphi(g)\overline{\theta_{4k+2l+3}(h,g)}\,dg,
\]
we let
\begin{equation}\label{eq:1}
F_{\psi_{a,V_{\calO'}}}(f)(h)=\int_{[\SO_{2k+1}]}\int_{[V_{\calO'}]}\varphi(g)\overline{\theta_{4k+2l+3}(vh,g)}\psi_{a,V_{\calO'}}(v)\,dvdg.
\end{equation}

The maximal split torus of $\SO_{4k+2l+3}(F)$ normalizes $V_{\calO'}$. Conjugating the variable $v$ in the inner integration of (\ref{eq:1}) by  $\tau=\diag(t,\cdots,t,1,t^{-1},\cdots, t^{-1})), t\in F^\times $ leaves the integral unchanged. The automorphicity of $\theta_{4k+2l+3}$ implies that it is left invariant by $\tau$.  After a change of variables by $v\mapsto v\tau^{-1}$, we see that the Fourier coefficient depends only on the square class of $a$ in $F^\times$. When $a$ is a square, the connected component of the stabilizer of $\psi_{a,V_{\calO'}}$  in $M(\calO')(F)$ is split. In this case, we call $\psi_{a,V_{\calO'}}$ a split generic character, and denote it by $\psi_{\calO'}$.

On the other hand, recall  $U_\calO=V_{1,\calO} = V_{2,\calO}$ is the unipotent radical of the maximal parabolic subgroup $P_\calO$ of $\SO_{2k+1}$ with the corresponding Levi part 
\[
M(\calO) = \GL^{n_1-n_2}_{s_1}\times \GL^{n_2-n_3}_{s_2}\times \cdots\times \GL^{n_{p-1}-n_p}_{s_{p-1}}\times\GL_l^{n_p}\times  \SO_{l}.\]
In order to define a generic character on $[U_\calO]$, it suffices to specify the respective character on each of the components of the maximal abelian quotient
\begin{equation}
	U_{\calO}/U_{\calO}^{(1)} \cong \Bigg(\bigoplus_{j=1}^{p-1} \Mat_{s_j\times s_{j+1}}\Bigg)\oplus \Bigg(\bigoplus_{j=1}^{p-1} \Mat_{s_j\times s_j}^{n_j-n_{j+1}-1}\Bigg)\oplus \Mat_{l\times l}^{n_p-1}.
\end{equation}
We define these characters by
\begin{equation}\label{charx}\begin{cases}
		\psi_{s_j\times s_j}(v) =\psi(\Tr v)  &\text{if}\;v\in \Mat_{s_j\times s_j}(\bbA), j=1,2,\cdots, p, \\
		\psi_{s_j\times s_{j+1}}(v) =\psi(\Tr (I_{s_j,s_{j+1}} v^T))& \text{if}\; v\in \Mat_{s_j\times s_{j+1}}(\bbA), j=1,2,\cdots, p-1.
	\end{cases}
\end{equation}
This gives a generic character $\psi_{\calO}:[U_\calO]\rightarrow \bbC^\times $ .

\begin{theorem}\label{Thm1}
	Let $(\pi, \calV)$ be an irreducible cuspidal genuine automorphic representation of $\wt{\SO}_{2k+1}(\bbA)$. Suppose the theta lifting $\Theta(\pi)$, as a representation of $\wt{\SO}_{2k+2l+2}(\bbA)$, has a non-zero Fourier coefficient with respect to a split generic character associated with the unipotent orbit 
	\[
	\calO' =\left((2n_1+3)^{r_1}(2n_2+3)^{r_2}\cdots(2n_p+3)^{r_p}(1)\right).
	\]
	Then the representation $\pi$ has a non-zero Fourier coefficient with respect to some generic character associated with the unipotent orbit
	\[
	\calO =\left((2n_1+1)^{r_1}(2n_2+1)^{r_2}\cdots(2n_p+1)^{r_p}\right).
	\]
	\end{theorem}

\begin{proof}
Throughout the proof, we identify any subgroup of $\SO_{2k+1}$ or $\SO_{2k+2l+2}$ with its embedded image in $\SO_{4k+2l+3}$ via \ref{ebd}. The Fourier coefficient of $\Theta(\pi)$ with respect to a generic character depends only on the square class of $a$. Therefore, we may assume there exists data such that the following integral is non-vanishing:
\begin{equation}\label{eq:2}
F_{\psi_{\calO'}}(f)(1)=\int_{[\SO_{2k+1}]}\int_{[V_{\calO'}]}\varphi(g)\overline{\theta}_{4k+2l+3}(v,g)\psi_{\calO'}(v)\,dvdg.
\end{equation}

Denote by $R_{s_1} =R_{s_1,4k+2l+3}$ the unipotent radical of the maximal parabolic subgroup of $\SO_{4k+2l+3}$ with Levi subgroup $\GL_{s_1}\times \SO_{4k+2l-2s_1+3}$. Notice that $V_{s_1}=V_{\calO'}\cap R_{s_1} $ is non-trivial. The quotient $V_{s_1}\setminus R_{s_1}$ may be identified with the subgroup of matrices of the form
\[
H_{s_1}=H_{s_1,4k+2l+3}: =\left\{\begin{pmatrix}
I_{s_1}&&x&&\ast\\
&I_{k+l-s_1+1}&&&\\
&&I_{2k+1}&&x^\ast\\
&&&I_{k+l-s_1+1}&\\
&&&&I_{s_1}
\end{pmatrix}\in\SO_{4k+2l+3} :x\in \Mat_{s_1\times (2k+1)}\right\}.
\]
Although $H_{s_1}$ is not abelian, it is a Heisenberg group with $Z(H_{s_1})$ corresponding to matrices of the form
	\[
	\left\{\begin{pmatrix}
	I_{s_1}&0&z\\
	&I_{4k+2l-2s_1+3}&0\\
	&&I_{s_1}
	\end{pmatrix}:z^TJ_{s_1}+J_{s_1}z =0\right\}.
	\]
Notice that the center $Z(H_{s_1})(\bbA)\subset V_{\calO'}(\bbA)$ is included in the domain of integration in (\ref{eq:2}). As a result, we expand the integral (\ref{eq:2}) against the abelian quotient $\left[\left(H_{s_1}/Z(H_{s_1})\right)\right] \cong [\Mat_{s_1\times (2k+1)}]$.  
	
We may identify the character group of $\left[\left(H_{s_1}/Z(H_{s_1})\right)\right]$ with  $\Mat_{s_1\times (2k+1)}(F)$. We claim that the only non-zero contributions from the expansion are those characters corresponding to matrices in $\Mat_{s_1\times (2k+1)}(F)$ with maximal rank and totally-isotropic row space.  
	
First, we look at the contribution from the trivial character, which corresponds to the zero matrix in $\Mat_{s_1\times (2k+1)}(F)$. This is the constant term of the Fourier expansion given by
\begin{equation}	\int_{[\SO_{2k+1}]}\int_{[V_{\calO'}]}\int_{[H_{s_1}/Z(H_{s_1})]}\varphi(g)\overline{\theta}_{4k+2l+3}\left(u(v,g)\right)\psi_{\calO'}(v)\,dudvdg.	\end{equation}
Write $v =v_{s_1}v^1$ where $v_{s_1}\in V_{s_1}(\bbA)$ and $v^1\in V_{\calO'}^1(\bbA) := V_{\calO'}(\bbA)\cap \SO_{2k+2l-2s_1+2}(\bbA)$. Combining the two variables $u$ and $v_{s_1}$, we obtain
\begin{equation}\label{eq:3}
		\int_{[\SO_{2k+1}]}\int_{[V_{\calO'}^1]}\int_{[R_{s_1}]}\varphi(g)\overline{\theta_{4k+2l+3}\left(u(v^1,g)\right)}\psi_1(u)\psi_{\calO'}(v^1)\,dudv^1dg,
	\end{equation}
	where $\psi_1$ is the character  on $[R_{s_1}]$ given by
	\begin{equation}
		\psi_1(u) = \psi(u_{1,s+1}+u_{2,s+2}+\cdots+ u_{s+2s}).
	\end{equation}
	 Applying \hyperref[prop3]{Proposition \ref*{prop3}}, we see that (\ref{eq:3}) is equal to the integral
	\begin{equation}\label{eq:4}
		\int_{[\SO_{2k+1}]}\int_{[V_{\calO'}^1]}\int_{[R_{s^2_1}]}\varphi(g)\overline{\theta}_{4k+2l+3}\left(u(v^1,g)\right)\psi_1(u)\psi_{\calO'}(v^1)\,dudv^1dg,
	\end{equation}
	where we extend $\psi_1$ trivially to $[R_{s^2_1}]$ and $R_{s^2_1} = R_{s_1,4k+2l+3}R_{s_1,4k+2l-2s_1+3}$ is the unipotent radical of the maximal parabolic subgroup of $\SO_{4k+2l+3}$ with Levi subgroup
	\[
	\GL_{s_1}\times \GL_{s_1}\times\SO_{4k+2l-4s_1+3}.
	\] 
	Notice that $R_{s^2_1}\cap V_{\calO'}^1$ is non-trivial. Let $\beta$  be the root inside $\SO_{2k+2l+2}$ such that
\begin{equation}\label{root}
\beta= 	\begin{cases}
		\sum_{j=1}^{s_2} \alpha_{s_1+j}&\text{if}\; n_1-n_2=1\\
		\sum_{j=1}^{s_1} \alpha_{s_1+j}&\text{if}\; n_1-n_2>1,
		\end{cases}
		\end{equation}
		where we recall that $\alpha_i$'s are the positive simple roots of $\SO_{2k+2l+2}$. By construction, the one parameter subgroup $\{x_\beta(r): r\in\bbA\}$ associated to $\beta$  is in the intersection $R_{s^2_1}(\bbA)\cap V_{\calO'}^1(\bbA)$, and $\psi_{\calO'}$ is non-trivial on $x_\beta(r)$.  We may write  (\ref{eq:4}) as 
			\begin{multline*}
	\int_{[\SO_{2k+1}]}\varphi(g)\int_{V_{\calO'}^1(F)x_\beta(\bbA)\setminus V_{\calO'}^1(\bbA)}\int_{[R_{s^2_1}]}\int_{\bbA/F}\overline{\theta}_{4k+2l+3}\left(ux_\beta(r)(v^1,g)\right)\psi_1(u)\psi(r)\psi_{\calO'}(v^1)\,drdudv^1dg\\
		=\left(\int_{\bbA/F}\psi(r) \,dr\right)\int_{[\SO_{2k+1}]}\varphi(g)\int_{V_{\calO'}^1(F)x_\beta(\bbA)\setminus V_{\calO'}^1(\bbA)}\int_{[R_{s^2_1}]}\overline{\theta}_{4k+2l+3}\left(u(v^1,g)\right)\psi_1(u)\psi_{\calO'}(v^1)\,dudv^1dg.
		\end{multline*}
	This is zero since
		\[
		\int_{\bbA/F} \psi (r) \,dr =0 
		\]
		for the non-trivial character $\psi$.
		
Second, we look at the contributions from the non-trivial characters. We identify every non-trivial character $\psi_\xi$ of $[\left(H_{s_1}/Z(H_{s_1})\right)]$ with some non-zero $\xi \in \Mat_{s_1\times (2k+1)}(F)$. The conjugation action by the diagonally embedded group $\GL_{s_1}(F)\times \SO_{2k+1}(F)$ on $[\left(H_{s_1}/Z(H_{s_1})\right)]$ induces an action on $\Mat_{s_1\times (2k+1)}(F)$. If $\xi$ contains any row vector in $F^{2k+1}$ that is non-isotropic, then the corresponding contribution is the integral 
\begin{equation}\label{eq:5}
	\int_{[\SO_{2k+1}]}\int_{[V_{\calO'}^1]}\int_{[R_{s_1}]}\varphi(g)\overline{\theta}_{4k+2l+3}\left(u(v^1,g)\right)\psi_{1,\xi}(u)\psi_{\calO'}(v^1)\,dudv^1dg.
	\end{equation}
Here, for any $u = u_1u_2\in R_{s_1}(\bbA)$ with $ u_1\in V_{s_1}(\bbA),u_2\in H_{s_1}(\bbA)$, $\psi_{1,\xi}(u) = \psi_1(u_1)\psi_\xi(u_2)$ is a generic character associated with the unipotent orbit corresponding to the partition $(3^{s_1}1^{4k+2l+3-3s_1})$. Hence, (\ref{eq:5}) contains a Fourier coefficient of $\theta_{4k+2l+3}$ associated to the unipotent orbit $\calO_\xi = (3^{s_1}1^{4k+2l+3-3s_1})$, which is zero by \hyperref[prop2]{Proposition \ref*{prop2}}. 
			
Under the $\GL_{s_1}(F)\times \SO_{2k+1}(F)$ action, any $\xi\in \Mat_{s_1\times (2k+1)}(F)$ whose row space is not totally isotropic lies in the same orbit as the non-isotropic ones with the same rank. The same argument shows that the contribution of a character corresponding to any such $\xi$ is zero. 
			
Thus, the only possible non-zero contributions are from $\psi_\xi$ with the row space of the matrix $\xi$ being totally isotropic . All such matrices of a given rank lie in the same orbit under the $\GL_{s_1}(F)\times \SO_{2k+1}(F)$-action. Thus, we may pick the representatives as
		\begin{equation}\label{psixi}
		\xi_q = \begin{pmatrix}
		I_q&0\\
		0&0
		\end{pmatrix}\in \Mat_{s_1\times (2k+1)}(F), \quad q =1,2,\cdots, s_1.
		\end{equation}
		For a given $\xi_q$, the contribution of the character $\psi_{\xi_q}:[\left(H_{s_1}/Z(H_{s_1})\right)]\rightarrow \bbC^\times $ is 
		\begin{equation}\label{eq:6}
	\int_{[\SO_{2k+1}]}\int_{[V_{\calO'}^1]}\int_{[R_{s_1}]}\varphi(g)\overline{\theta}_{4k+2l+3}\left(u(v^1,g)\right)\psi_{1,\xi_q}(u)\psi_{\calO'}(v^1)\,dudv^1dg,
	\end{equation}
	where $\psi_{1,\xi_q}= \psi_1\psi_{\xi_q}$ is the character on $[R_{s_1}/Z(R_{s_1})]$ that corresponds to the matrix 
	\[
	\begin{pmatrix}
I_{s_1}&0_{s_1\times (k+l+1-2s_1)}&\xi&0_{s_1\times (k+l-s_1+1)}
	\end{pmatrix} \in \Mat_{s_1\times (4k+2l-2s_1+3)}(F).
	\]
	
	Let $z_q\in \SO_{4k+2l+3}(F)$  be the unipotent element
	\[
	z_q = \begin{pmatrix}
	I_{s_1}&&&&\\
&	\mu_q&&&\\
&	&I_{2k+1-2q}&&\\
	&&&\mu_q^\ast&\\
	&&&&I_{s_1}
	\end{pmatrix}, \mu_q=\begin{pmatrix}
	I_q&&-I_q\\
	&I_{k+l+1-s_1-q}&\\
	&&I_q
	\end{pmatrix}.
	\]
Performing a change of variables by $u \mapsto uz_q$,  (\ref{eq:6}) is equal to
	\begin{equation}\label{eq:7}
	\int_{[\SO_{2k+1}]}\int_{[V_{\calO'}^1]}\int_{[R_{s_1}]}\varphi(g)\overline{\theta}_{4k+2l+3}\left(uz_q(v^1,g)\right)\psi'_{1,\xi_q}(u)\psi_{\calO'}(v^1)\,dudv^1dg,
	\end{equation}
	where $\psi'_{1,\xi_q}$ is the character on $[\left(H_{s_1}/Z(H_{s_1})\right)]$ that corresponds to the matrix 
	\[
	\begin{pmatrix}
	I^\ast_{s_1-q}&0_{s_1\times (k+l+1-2s_1)}&\xi&0_{s_1\times (k+l-s+1)}
	\end{pmatrix}\in \Mat_{s_1\times (4k+2l-2s_1+3)}(F), 
	\]
	\[
	I^\ast_{s_1-q} = \begin{pmatrix}
	0&0\\
	0&I_{s_1-q}
	\end{pmatrix}\in \Mat_{s_1\times s_1}(F).
	\]

Let $\omega_q$ be the Weyl group element given by
	\[
	\omega_q=\begin{pmatrix}
	I_{s_1}&&&&\\
	&\nu_q&&&\\
	&&I_{2k+1-2q}&&\\
	&&&\nu_q^{-1}&\\
	&&&&I_{s_1}
	\end{pmatrix}\in \SO_{4k+2l+3}(F), \nu_q = \begin{pmatrix}
	&&&I_q&\\
&I_{s_1-q}&&&\\
I_q&&&&\\
&&I_{k+l+q+1-3s_1}&&\\
&&&&I_{s_1-q}
	\end{pmatrix}.
	\]
	The conjugation action of $\omega_q$ stabilizes $[R_{s_1}]$.  Therefore, we change the variable $u\mapsto \omega^{-1}_q u\omega_q$ and use the automorphicity of $\theta_{4k+2l+3}$ to obtain 
		\begin{equation}\label{eq:8}
	\int_{[\SO_{2k+1}]}\int_{[V_{\calO'}^1]}\int_{[R_{s_1}]}\varphi(g)\overline{\theta}_{4k+2l+3}\left(u\omega_qz_q(v^1,g)\right)\psi_1(u)\psi_{\calO'}(v^1)\,dudv^1dg.
	\end{equation}
Apply  \hyperref[prop3]{Proposition \ref*{prop3}} to the integral (\ref{eq:8}) to replace the integration on $[R_{s_1}]$ by $[R_{s^2_1}]$. We obtain
		\begin{equation}\label{eq:9}
	\int_{[\SO_{2k+1}]}\int_{[V_{\calO'}^1]}\int_{[R_{s^2_1}]}\varphi(g)\overline{\theta}_{4k+2l+3}\left(u\omega_qz_q(v^1,g)\right)\psi_1(u)\psi_{\calO'}(v^1)\,dudv^1dg.
	\end{equation}
	This integral is similar to  (\ref{eq:4}) except for the presence of the Weyl group element $\omega_q$.  However, the same argument implies that the contribution is zero as long as $R_{s^2_1} \cap \omega_q V_{\calO'}^1\omega_q^{-1}$ is non-trivial. This happens only when $\Rank(\xi_q)<s_1$. Hence, we conclude that the only non-zero contributions of the Fourier expansion are from characters $\psi_\xi$ on $[\left(H_{s_1}/Z(H_{s_1})\right)]$ corresponding to some  $\xi\in\Mat_{s_1\times (2k+1)}(F)$ with rank $s_1$ and totally isotropic row space. Notice that $\SO_{2k+1}(F)$ acts transitively from the right on such matrices. If we take 
	\[
	\xi_{s_1} = \begin{pmatrix}
	I_{s_1} &0
	\end{pmatrix} \in \Mat_{s_1\times (2k+1)}(F)
	\]
as a representative, then we conclude that the Fourier coefficient (\ref{eq:2}) is equal to 
\begin{equation}\label{eq:10}
\int_{[\SO_{2k+1}]}\int_{[V_{\calO'}^1]}\int_{[R_{s^2_1}]}\sum_{\gamma\in P_{s_1}^0(F)\setminus\SO_{2k+1}(F)}\varphi(g)\overline{\theta}_{4k+2l+3}\left(u\omega_{s_1}z_{s_1}\gamma(v^1,g)\right)\psi_1(u)\psi_{\calO}(v^1)\,dudv^1dg.
\end{equation}
Here, $P_{s_1}=P_{s_1,2k+1}$ is the standard maximal parabolic subgroup of $\SO_{2k+1}$ with Levi part  $\GL_{s_1}\times \SO_{2k-2s_1+1}$. The upper zero indicates that we omit the $\GL_{s_1}$ factor. In fact, $P^0_{s_1}(F)$ is the stabilizer of $\psi_{\xi_{s_1}}$ under the $\SO_{2k+1}(F)$ action. As any $\gamma\in P_{s_1}^0(F)\setminus\SO_{2k+1}(F)$ commutes with any $v^1\in V_{\calO'}^1(\bbA)$, we can combine the summation with the integration in (\ref{eq:10}) to rewrite it as 
\begin{equation}\label{eq:11}
\int_{P^0_{s_1}(F)\setminus \SO_{2k+1}(\bbA)}\int_{[V_{\calO'}^1]}\int_{[R_{s^2_1}]}\varphi(g)\overline{\theta}_{4k+2l+3}\left(u\omega_{s_1}z_{s_1}(v^1,g)\right)\psi_1(u)\psi_{\calO'}(v^1)\,dudv^1dg.
\end{equation}

To proceed, we start with the assumption that $n_1-n_2=1$ for notational simplicity. The same argument applies to the other case. Consider the unipotent subgroup $R_{s_2}= R_{s_2,4k+2l-4s_1+3}\subset \SO_{4k+2l+3}$, which is the unipotent radical of the standard maximal parabolic subgroup whose Levi part is $\GL_{s_1}^2 \times \GL_{s_2}\times \SO_{4k+2l-4s_1-2s_2+3}$. The group $R_{s_2}$ consists of matrices of the form
\begin{equation}\label{rs2}
\left\{\begin{pmatrix}
I_{2s_1}&&&&\\
&I_{s_2}&x&\ast&\\
&&I_{4k+2l-4s_1-2s_2+3}&x^\ast&\\
&&&I_{s_2}&\\
&&&&I_{2s_1}
\end{pmatrix}\in \SO_{4k+2l+3}: x\in \Mat_{s_2\times (4k+2l-4s_1-2s_2+3)}\right\}.
\end{equation}
We set $V_{s_2} = R_{s_2} \cap \left(\omega_{s_1}z_{s_1}V_{\calO'}^1(\omega_{s_1}z_{s_1})^{-1}\right)$, and the quotient $H_{s_2} = V_{s_2}\setminus R_{s_2}$. As the center $[Z(H_{s_2})]\subset [V_{s_2}]$ is included in the domain of integration in (\ref{eq:11}), we continue to expand (\ref{eq:11}) against $[H_{s_2}/Z(H_{s_2})]$.  We check on the contributions from each type of the characters on $[H_{s_2}/Z(H_{s_2})]$ under the action of $\GL_{s_2}(F)\times \SO_{4k+2l-4s_1-2s_2+3}(F)$. By the same argument, it follows that the only contribution is from the orbit of characters represented by $\psi_{\xi_{s_2}}$ corresponding to the matrix 
\[
\xi_{s_2} = \begin{pmatrix}
	I_{s_2} &0
\end{pmatrix}\in \Mat_{s_2\times (4k+2l-4s_1-2s_2+3)}(F).
\]

We continue the same argument repeatedly with the assumption that $n_i-n_{i+1} =1$ for all $i =1,2,\cdots, p-1$ and $n_p=1$ for notational simplicity. We deduce that (\ref{eq:11}) equals to
\begin{equation}\label{eq:12}
\int_{P_{\calO}^0(F)\setminus \SO_{2k+1}(\bbA)}\int_{[V_{\calO'}^p]}\int_{[R_{s^2_p}]}\varphi(g)\overline{\theta}_{4k+2l+3}\left(u\prod_{i=0}^{p-1}\omega_{s_{p-i}}z_{s_{p-i}}(v^p,g)\right)\psi_p(u)\psi_{\calO'}(v^p)\,dudv^pdg.
\end{equation}

We explain the notations here. First, $P_{\calO} =M(\calO)U_\calO$ is the maximal parabolic subgroup of $\SO_{2k+1}$ with the corresponding Levi decomposition. The upper zero indicates that we omit all the $\GL$-factors in the Levi factor $M(\calO)$. That is $P_{\calO}^0 \cong \SO_lU_\calO$. Next, $V_{\calO'}^p = V_{\calO'}\cap \SO_{3l+1}$ with $V_{\calO'}^p/(V_{\calO'}^p)^{(1)}\cong \Mat_{l\times (l+1)}$. Also, $R_{s^2_p}$ is the unipotent radical of the standard maximal parabolic subgroup of $\SO_{4k+2l+3}$ with Levi part given by
\[
\GL^{2(n_1-n_2)}_{s_1}\times \GL^{2(n_2-n_3)}_{s_2}\times \cdots\times \GL^{2(n_{p-1}-n_p)}_{s_{p-1}}\times \GL^{2n_p}_l\times \SO_{4l+1}.
\]
The character $\psi_p:[R_{s^2_p}]\rightarrow \bbC^\times$ is the product of $\psi_1$ and the characters corresponding to the non-zero contribution in each of the repeated steps. It is given by (similar to $\psi_n$ in \hyperref[cor]{Corollary \ref*{cor}})
\begin{equation*}\label{psip}
\psi_p(u) = \psi\left(\sum_{i=1}^{s_1}u_{j,s_1+j}+\sum_{j=1}^{s_2}u_{2s_1+j,2s_1+s_2+j}+
\cdots +\sum_{j=1}^{s_p}u_{2(s_1+\cdots+s_{p-1})+j,2(s_1+\cdots+s_{p-1})+s_p+j}\right).
\end{equation*}
The term $\prod_{i=0}^{p-1}\omega_{s_{p-i}}z_{s_{p-i}}$ is the product of  $\omega_{s_1}z_{s_1}$ and the corresponding Weyl group elements and unipotent elements produced during each of the repeated steps. To be more precise, for each $1\leqslant j\leqslant p$, we have
\[
\omega_{s_j} = \begin{pmatrix}
I_{2(s_1+\cdots+s_{j-1})}&&&&\\
&\nu_{s_j}&&&\\
&&I_{2k-2(s_1+\cdots+s_j)+1}&&\\
&&&\nu_{s_j}^{-1}&\\
&&&&I_{2(s_1+\cdots+s_{j-1})}
	\end{pmatrix},
	\]
	
	\[ \nu_{s_j} =\begin{pmatrix}
	I_{s_j}&&\\
	&&I_{s_j}\\
	&I_{k+l+1-2(s_1+\cdots+s_{j-1})-s_j}&
	\end{pmatrix},
	\]
and
	\[
	z_{s_j} =  \begin{pmatrix}
	I_{2(s_1+\cdots+s_{j-1})}&&&&\\
	&\mu_{s_j}&&&\\
	&&I_{2k-2(s_1+\cdots+s_j)+1}&&\\
	&&&\mu_{s_j}^{\ast}&\\
	&&&&I_{2(s_1+\cdots+s_{j-1})}
	\end{pmatrix},
	\]
	\[
	\mu_{s_j}=\begin{pmatrix}
	I_{s_j}&&&&\\
	&I_{s_j}&&-I_{s_j}\\
	&&I_{k+l+1-2(s_1+\cdots+s_j)}&\\
	&&&I_{s_j}
	\end{pmatrix}.
	\]
Notice that we can conjugate all the $z_j$'s to the right of all the $\omega_j$'s and rewrite (\ref{eq:12}) as 
 \begin{equation}\label{eq:13}
 \int_{P_{\calO}^0(F)\setminus \SO_{2k+1}(\bbA)}\int_{[V_{\calO'}^p]}\int_{[R_{s^2_p}]}\varphi(g)\overline{\theta}_{4k+2l+3}\left(u\omega_lz_l(v^p,g)\right)\psi_p(u)\psi_{\calO'}(v^p)\,dudv^pdg,
 \end{equation}
 with $\omega_l =\omega_{s_{p-1}}\omega_{s_{p-2}}\cdots\omega_{s_1}$ and 
 \[
 z_l =  \begin{pmatrix}
 I_{s_1}&&&&\\
 &\mu_l&&&\\
 &&I_l&&\\
 &&&\mu_l^{-1}&\\
 &&&&I_{s_1}
 \end{pmatrix}, \mu_l=\begin{pmatrix}
 I_{s_1+\cdots+s_{p-1}}&&-I_{s_1+\cdots+s_{p-1}}\\
&I_{k+l-2(s_1+\cdots+s_{p-1})+1}&\\
&&I_{s_1+\cdots+s_{p-1}}
 \end{pmatrix}.
 \]
 
 Now we proceed to the last step. Let $R_p\subset \SO_{4l+1}$ be the unipotent radical of the standard maximal parabolic subgroup of $\SO_{4l+1}$ with Levi part $\GL_{l}\times \SO_{2l+1}$. In terms of matrices, 
 \[
 R_p = \left\{\begin{pmatrix}
 I_{2(s_1+\cdots+s_{p-1})}&&&&\\
 &I_{l}&x&\ast&\\
 &&I_{2l+1}&x^\ast&\\
 &&&I_{l}&\\
 &&&&I_{2(s_1+\cdots+s_{p-1})}
 \end{pmatrix} \in \SO_{4k+2l+3}: x\in \Mat_{l\times (2l+1)}\right\}.
 \]
The subgroup $R_p$ is again a Heisenberg group with the maximal abelian quotient $R_p/R^{(1)}_p\cong \Mat_{l\times (2l+1)}$. On the other hand, for any $y=\begin{pmatrix}
y_1&y_2
\end{pmatrix}\in V_{\calO'}^p/(V_{\calO'}^p)^{(1)}\cong \Mat_{l\times (l+1)}$ where $y_1,y_2\in \Mat_{l\times (\frac{l+1}{2})}$, we have 
\[
\omega_lz_l y(\omega_lz_l)^{-1} =\begin{pmatrix}
y_1&0_{l\times l}&y_2
\end{pmatrix}\in R_p/R_p^{(1)}.
\]
Let $H_p  = \left(\omega_lz_l V^p(\omega_lz_l)^{-1}\right) \setminus R_p$, which is also a Heisenberg group. As the center $[Z(H_p)]\subset  [\left(\omega_lz_l V^p(\omega_lz_l)^{-1}\right)]$ is included in the domain of integration in the integral (\ref{eq:13}), we further expand (\ref{eq:13})  against $[\left(H_p/Z(H_p)\right)]$. 

We may identify the character group of $[\left(H_p/Z(H_p)\right)]$ with 
\[
\left(H_p/Z(H_p)\right)(F)\cong \Mat_{l\times l}(F).
\]
In the expansion, we conjugate elements in $[V^p_{\calO'}]$ to the left inside the function $\theta_{4k+2l+3}$ and combine this domain of integration with $[\left(H_p/Z(H_p)\right)]$. Denote the resulting domain of integration by $[R_p]$. Let $R=R_{s^2_p}R_p$, which is the unipotent subgroup that coincides with $V_{2,\calO^2}$ where $\calO^2$ is the unipotent orbit in $\SO_{4k+2l+3}$ corresponding to the partition 
\[
\left((4n_1+3)^{r_1}(4n_2+3)^{r_2}\cdots (4n_p+3)^{r_p}(1)^{l+1}\right).
\] As a result, (\ref{eq:13}) is equal to
 \begin{equation}\label{eq:14}
\int_{P_{\calO}^0(F)\setminus \SO_{2k+1}(\bbA)}\int_{[R]}\sum_{\xi\in\Mat_{l\times l}(F)}\varphi(g)\overline{\theta_{4k+2l+3}\left(u\omega_lz_l(1,g)\right)}\psi_{\calO',p,\xi}(u)\,dudg.
\end{equation}
Here, we combine all the characters involved to get the character $\psi_{\calO',p,\xi}$ on $[R]$. For any $x=x_1x_2x_3\in R(\bbA)$ with $x_1\in R_{s^2_p}(\bbA),x_2\in V_{\calO'}^p(\bbA)$ and $x_3\in H_p(\bbA)/(Z(H_p))(\bbA)$, 
	\[
	\psi_{\calO',p,\xi} (x) = \psi_p(x_1)\psi_{\calO'}(x_2)\psi_\xi(x_3).
	\]
	
	By \hyperref[prop1]{Proposition \ref*{prop1}}, the inner  integral of (\ref{eq:14}) factors as the product of an integral of a theta function in the theta representation $\Theta_{\GL}$ of the double cover $\wt{\GL}_{2(s_1+\cdots+s_{p-1})}(\bbA)$ with respect to the character $\psi_p$ and an integral of a theta function in $\Theta_{4l+1}$ of $\wt{\SO}_{4l+1}(\bbA)$ with respect to the character $\psi_{\calO'}\psi_\xi$. When the character $\psi_{\calO'}\psi_\xi$ is generic with respect to the unipotent orbit associated with the partition $(3^l1^{l+1})$, the second integral is a Fourier coefficient of the theta representation $\Theta_{4l+1}$ with respect to the unipotent orbit associated to $(3^l1^{l+1})$. By \hyperref[prop2]{Proposition \ref*{prop2}}, such an integral must be zero. Thus, the non-zero contributions in the expansion (\ref{eq:14}) come from those $\psi_\xi$ corresponding to $\xi\in \Mat_{l\times l}(F)$ such that $\psi_{\calO'}\psi_\xi$ is not generic.

   To proceed, let us fix a
    \[
    \xi =\begin{pmatrix}
    \lambda_1\\
    \vdots\\
   \lambda_l
    \end{pmatrix} \in\Mat_{l\times l}(F),
    \]
   with row vectors $\lambda_1,\cdots,\lambda_l\in F^l$.  By construction, the character $\psi_\eta :=\psi_{\calO'}\psi_\xi$ corresponds to the matrix
    \[
    \eta =\begin{pmatrix}
    1&&&&\lambda_1&&&&\\
    &\ddots&&&\vdots&&&&\\
    &&1&&\lambda_{\frac{l-1}{2}}&&&&\\
    &&&\frac{1}{2}&\lambda_{\frac{l+1}{2}}&-1&&&\\
    &&&&\lambda_{\frac{l+3}{2}}&&-1&&\\
    &&&&\vdots&&&\ddots&\\
    &&&&\lambda_l&&&&-1
    \end{pmatrix}\in\Mat_{l\times (2l+1)}(F).
    \]
    Notice that $\psi_\eta$ is generic as long as the row space of $\eta$ is not totally isotropic with respect to the non-degenerate bilinear form on $F^{2l+1}$ defined by $J_{2l+1}$. As a result, the contribution of $\psi_\eta$ is non-zero only if the row space of $\eta$ is totally isotropic and of rank $l$. That is, the row space of $\eta$ is a maximal totally isotropic subspace of $F^{2l+1}$. Consequently, we only need to consider $\xi$ with row vectors $\lambda_1, \lambda_2,\cdots \lambda_l$ satisfying  the following conditions:
    \begin{enumerate}
    	\item Each $\lambda_i$ except $\lambda_{\frac{l+1}{2}}$ is non-zero isotropic.\\
    	\item Each pair $\lambda_i$ and $\lambda_{l+1-i}$ for $i=1,2,\cdots, \frac{l-1}{2}$ is a hyperbolic pair, i.e.
    	\[
    	(\lambda_i,\lambda_{l+1-i}) = 1,\;i=1,2,\cdots,\frac{l-1}{2}.
    	\]
    	All these hyperbolic pairs are mutually orthogonal.
    	\item  The vector $\lambda_{\frac{l+1}{2}}$ is non-isotropic with $(\lambda_{\frac{l+1}{2}},\lambda_{\frac{l+1}{2}})=1$. \
    \end{enumerate}
The group $\SO_{l}(F)$ acts transitively from the right on the set of $\xi$'s and we  pick the identity matrix $I_{l}\in \Mat_{l\times l}(F)$ as a representative. As a result, the integral (\ref{eq:14}) is equal to
\begin{equation}\label{eq:15}
	\int_{P_{\calO}^0(F)\setminus \SO_{2k+1}(\bbA)}\int_{[R]}\sum_{\gamma\in\SO_{2l+1}(F)}\varphi(g)\overline{\theta}_{4k+2l+3}\left(u\gamma\omega_lz_l(1,g)\right)\psi_{\calO',p,I_{l}}(u)\,dudg.
\end{equation}
Since $\gamma $ commutes with the product $\omega_l z_l$, we can conjugate it to the right and collapse the summation with the outer integration to obtain  
	\begin{equation}\label{eq:16}
		\int_{U_{\calO}(F)\setminus \SO_{2k+1}(\bbA)}\int_{[R]}\varphi(g)\overline{\theta}_{4k+2l+3}\left(u\omega_lz_l(1,g)\right)\psi_{\calO',p,I_{l}}(u)\,dudg.
	\end{equation}
We further decompose the domain of integration of the outer integral as
\[
U_\calO(F)\setminus \SO_{2k+1}(\bbA) =\left(U_\calO(F)\setminus U_\calO(\bbA)\right)\left( U_\calO(\bbA)\setminus\SO_{2k+1}(\bbA)\right).
\]
We obtain 
\begin{equation}\label{eq:17}
\int_{U_\calO(F)\setminus U_\calO(\bbA)}\int_{U_\calO(\bbA)\setminus\SO_{2k+1}(\bbA)}\int_{[R]}\varphi(ng)\overline{\theta}_{4k+2l+3}\left(u\omega_lz_l(1,ng)\right)\psi_{\calO',p,I_{l}}(u)\,dudgdn.
\end{equation}
Observe that $n_0:= \omega_lz_l n(\omega_lz_l)^{-1} \in R(\bbA)$ for any $n\in U_\calO(\bbA)$, and $\psi_{\calO',p,I_{l}}(n_0) =\psi_\calO(n)$. Hence, we conjugate the variable $n$ inside the function $\theta_{4k+2l+3}$ to the left and perform a change of variable by $u \mapsto un_0^{-1} $ to finally obtain that (\ref{eq:17}) is equal to
\begin{equation}\label{eq:18}
\int_{U_\calO(\bbA)\setminus\SO_{2k+1}(\bbA)}\left(\int_{U_\calO(F)\setminus U_\calO(\bbA)}\varphi(ng)\psi_\calO(n)dn\right)\int_{[R]}\overline{\theta}_{4k+2l+3}\left(u\omega_lz_l(1,g)\right)\psi_{\calO',p,I_{l}}(u)\,dudg.
\end{equation}
Here $F_{\psi_\calO,U_\calO}(\varphi)(g) = \int_{U_\calO(F)\setminus U_\calO(\bbA)}\varphi(ng)\psi_\calO(n)dn$ is a Fourier coefficient of $\varphi$ associated to the unipotent orbit $\calO$ and the generic character $\phi_\calO$ on $[U_\calO]$. The fact that (\ref{eq:2}) is non-vanishing implies that $F_{\psi_\calO,U_\calO}(\varphi) $ is non-vanishing, which completes the proof.
	\end{proof}

\section{\textbf{Local theory}}\label{sec6}
In this section, we establish the local counterpart to \hyperref[Thm1]{Theorem \ref*{Thm1}}. Let $F$ be a non-archimedean local field. We are still concerned with the two unipotent orbits
\[
\calO =\left((2n_1+1)^{r_1}(2n_2+1)^{r_2}\cdots(2n_p+1)^{r_p}\right)
\]
and 
\[
\calO' =\left((2n_1+3)^{r_1}(2n_2+3)^{r_2}\cdots(2n_p+3)^{r_p}(1)\right)
\]
of the two groups $\SO_{2k+1}$ and $\SO_{2k+2l+2}$ respectively, where $l=r_1+r_2+\cdots r_p$. 

Recall that we can associate the two orbits $\calO$ and $\calO'$ with the unipotent subgroups $U_\calO=V_{1,\calO}= V_{2,\calO}$ and $V_{\calO'}=V_{1,\calO'} = V_{2,\calO'}$ respectively. Define the generic characters $\psi_\calO: U_{\calO}(F)\rightarrow \bbC^\times$ and $\psi_{\calO'}: V_{\calO'}(F)\rightarrow\bbC^\times$  similarly to the global ones given by (\ref{charv}) and (\ref{charx}) respectively.

Suppose $\pi$ is an irreducible genuine admissible representation of $\wt{\SO}_{2k+1}(F)$. Recall from \hyperref[sec41]{Section \ref*{sec41}} that $\Theta_{4k+2l+3}$ is the local theta representation of $\wt{\SO}_{4k+2l+3}(F)$. Suppose there exists an irreducible genuine admissible representation $\Theta(\pi)= \Theta_{4k+2l+3}(\pi)$ of $\wt{\SO}_{2k+2l+2}(F)$ such that the $\Hom$-space
\begin{equation}
	\Hom_{\wt{\SO}_{2k+1}\times \wt{\SO}_{2k+2l+2}}\left(\Theta_{4k+2l+3},\pi\otimes \Theta(\pi)\right) 	\end{equation}
is non-zero. Here, we restrict $\Theta_{4k+2l+3}$ to a representation of the product subgroup $\wt{\SO}_{2k+1}(F)\times \wt{\SO}_{2k+2l+2}(F)$ which is the preimage of the product $\SO_{2k+1}(F)\times \SO_{2k+2l+2}(F)$ in $\wt{\SO}_{4k+2l+3}(F)$ via the embedding (\ref{ebd}). 

We extend the definition of twisted Jacquet modules in \hyperref[sec33]{Section \ref*{sec33}}. If $\calO_0$ is an odd unipotent orbit, then $U:=V_{2,\calO_0}$ is equal to the unipotent radical of some maximal parabolic subgroup of $\SO_{2k+1}$. If the unipotent radical $U(F)$ is a Heisenberg group with its center $Z(U)(F)$ acting trivially on $\calV$ and $\psi_0$ is trivial on $Z(U)(F)$, then the vector subspace $\calV(U,\psi_0) = \calV(U/Z(U),\psi_0)$. Denote the quotient space $\calV /\calV(U/Z(U),\psi_0)$ by $J_{U/Z(U), \psi}(\pi)$. 


\begin{theorem}\label{Thm2}
	Let $\pi$ be an irreducible admissible representation of $\wt{\SO}_{2k+1}(F)$. Suppose there exists an irreducible admissible representation $\Theta(\pi)$ of $\wt{\SO}_{2k+2l+2}(F)$ such that, as representations of the group $\wt{\SO}_{2k+1}(F)\times \wt{\SO}_{2k+2l+2}(F)$,
	\begin{equation}\label{eq61}
		\Hom_{\wt{\SO}_{2k+1}\times \wt{\SO}_{2k+2l+2}}\left(\Theta_{4k+2l+3},\pi\otimes \Theta(\pi)\right) \neq 0.	\end{equation}
	Furthermore, suppose the twisted Jacquet module of $\Theta(\pi)$ with respect to the unipotent orbit $\calO'$ and the generic character  $\psi_{\calO'}$  is non-zero, i.e.
	\begin{equation}
		J_{V_{\calO'},\psi_{\calO'}}(\Theta(\pi)) \neq 0.
	\end{equation}
	Then the twisted Jacquet module of $\pi$ with respect to the unipotent orbit $\calO$ and the generic character $\psi_\calO$ is also non-zero, i.e.
	\begin{equation}
		J_{U_{\calO},\psi_\calO}(\pi) \neq 0.
	\end{equation}
\end{theorem}

In order to prove \hyperref[Thm2]{Theorem \ref*{Thm2}}, we need the following Proposition which is the local version of  \hyperref[prop3]{Proposition \ref*{prop3}}. For a positive integer $s<k/4$,  let $R_{s,2k+1}$ be the unipotent radical of the standard maximal parabolic subgroup of $\SO_{2k+1}$ with Levi factor $\GL_s\times \SO_{2k-2s+1}$. Similarly, let $R_{2s,2k+1}$ be the unipotent radical of the maximal parabolic subgroup of $\SO_{2k+1}$ with Levi factor $\GL_{2s}\times \SO_{2k-4s+1}$. Define the character $\psi_1:R_{s,2k+1}(F)\rightarrow \bbC^\times$ by
\[
\psi_1(u) =\psi(\sum_{j=1}^s u_{j,j+s}),\quad u=(u_{i,j})\in R_{s,2k+1}(F).\]
\begin{proposition}\label{prop63}
	Consider the local theta representation $\Theta_{2k+1}$ of $\wt{\SO}_{2k+1}(F)$.  There is a surjection of $\wt{\GL}_s^\Delta \times \wt{\SO}_{2k-4s+1}$ -modules 
	\[
	J_{R_{2s,2k+1}}(\Theta_{2k+1}) \twoheadrightarrow J_{R_{s,2k+1},\psi_1}(\Theta_{2k+1}),
	\]
	where $\GL_s^\Delta \times \SO_{2k-4s+1}$ is the subgroup of $\GL_{2s}\times \SO_{2k-4s+1}$ with the $\GL_s$-factor embedded in $\GL_{2s}$ diagonally.
\end{proposition}
\begin{proof}
	By \hyperref[prop61]{Proposition \ref*{prop61}}, it suffices to show that the unipotent subgroup
	\[
	R_{s,2k-2s+1}(F)\subset \SO_{2k-2s+1}(F) \xhookrightarrow{\iota} \SO_{2k+1}(F)
	\]  acts trivially on the Jacquet module $J_{R_{s,2k+1},\psi_1}(\Theta_{2k+1})$. The proof then proceeds in a similar fashion to the global case. 
	
	Let $R_{s,2k-2s+1}$ be the unipotent radical of the maximal parabolic subgroup of $\SO_{2k-2s+1}$ with the Levi factor $\GL_{s}\times \SO_{2k-4s+1}$. We identify $R_{s,2k-2s+1}$ with its embedded image in $\SO_{2k+1}$ via (\ref{ebd}). Note that $R_{s,2k-2s+1}(F)$ is a Heisenberg group. We claim that its center $Z(R_{s,2k-2s+1})(F)$ acts trivially on $J_{R_{s,2k+1},\psi_1}(\Theta_{2k+1})$. 

Under the action of the Levi subgroup, any non-trivial character on $Z(R_{s,2k-2s+1})(F)$ may be represented by $\psi_{\xi_t}$ ($0<2t\leqslant s$) associated to a matrix $\xi_t$ of the form
	\[
	\xi_t = \begin{pmatrix}
		0&z_t\\
		0&0
	\end{pmatrix}\in \Mat_{s\times s}(F),
	\]
	where $z_t = \diag(\lambda_1,\cdots,\lambda_t,\lambda_t^{-1},\cdots,\lambda_1^{-1})\in \Mat_{2t\times 2t}(F)$ with $\lambda_i\in F^\times, \,\forall i=1,\cdots,t$. If the claim is not true, then there must be a non-trivial character $\psi_{\xi_t}$ on $Z(R_{s,2k-2s+1})(F)$ such that 
	\begin{equation}\label{jac}
	J_{Z(R_{s,2k-2s+1}), \psi_{\xi_t}}\left(J_{R_{s,2k+1,},\psi_1}(\Theta_{2k+1})\right)\neq 0.
	\end{equation}
	However, the product of $\psi_1$ and  $\psi_{\xi_t,Z}$ is a generic character attached to the unipotent orbit $\calO_t=(4^{2t}3^{s-2t}1^{2k-2t-3s+1})$. Hence, the resulting twisted Jacquet module (\ref{jac}) is zero by \hyperref[prop62]{Proposition \ref*{prop62}}. We get a contradiction.
	
	Therefore, it remains to establish that the abelian quotient $R_{s,2k-2s+1}/Z(R_{s,2k-2s+1})(F)$ also acts trivially on $J_{R_{s,2k+1},\psi_1}(\Theta_{2k+1})$. We may identify the character group of $R_{s,2k-2s+1}/Z(R_{s,2k-2s+1})(F)$ with $\Mat_{s\times (2k-4s+1)}(F)$. 	Under the action of Levi subgroup $\GL_{s}(F)\times \SO_{2k-4s+1}(F)$, any character on  $R_{s,2k-2s+1}/Z(R_{s,2k-2s+1})(F)$ may be represented by $\psi_{2,z_t}$, corresponding to a matrix $z_t$ of the form
	\[
	z_t=\begin{pmatrix}
		I_t&0\\
		0&0
	\end{pmatrix} \in \Mat_{s\times (2k-4s+1)}(F), t=1,\cdots, s.
	\]
	If the action of  $R_{s,2k-2s+1}/Z(R_{s,2k-2s+1})(F)$ on $J_{R_{s,2k+1},\psi_1}(\Theta_{2k+1})$ is not trivial, then there exists a character $\psi_{2,z_t}$ such that \begin{equation}\label{jac2}
	J_{R_{s,2k-2s+1}/Z(R_{s,2k-2s+1}), \psi_{2,z_t}}\left(J_{Z(R_{s,2k-2s+1})}\left(J_{R_{s,2k+1,},\psi_1}(\Theta_{2k+1})\right)\right)\cong J_{R_{s^2,2k+1},\psi_1\psi_{2,z_t}}(\Theta) \neq 0,
	\end{equation}
	where $R_{s^2,2k+1}=R_{s,2k+1}R_{s,2k-2s+1}$. 
	
	Proceed with the twisted Jacquet module $J_{R_{s^2,2k+1},\psi_1\psi_{2,z_t}}(\Theta_{2k+1}) $. Apply the same argument to see that $Z(R_{s,2k-4s+1})(F)$ acts on it trivially.  Hence, it suffices to check the action of the abelian quotient $R_{s,2k-4s+1}/Z(R_{s,2k-4s+1})(F)$ on $J_{R_{s^2,2k+1},\psi_1\psi_{2,z_t}}(\Theta_{2k+1}) $. If the action is trivial, \hyperref[prop61]{Proposition \ref*{prop61}} implies that there is a $\wt{\GL}_{3s}(F)\times \wt{\SO}_{2k-6s+1}(F)$-module isomorphism 
	\begin{equation}
	J_{R_{s,2k-4s+1}}\left(J_{R_{s^2,2k+1},\psi_1\psi_{2,z_t}}(\Theta_{2k+1}) \right) \cong J_{N_{s^3}, \psi_1\psi_{2,z_t}}\left(\Theta_{\GL_{3s}(F)} \right)\otimes \Theta_{2k-6s+1}.
	\end{equation}
	Here $\Theta_{\GL_{3s}(F)} $ is the local theta representation of the double cover $\wt{\GL}_{3s}(F)$, and $N_{s^3}$ is the unipotent radical of the parabolic subgroup of $\GL_{3s}$ with Levi subgroup $\GL_{s}^3$.  However, by Corollary 3.34 of \cite{CAI19}, the twisted Jacquet module $J_{N_{s^3}, \psi_1\psi_{2,t}}\left(\Theta_{\GL_{3s}(F)} \right)$ is zero. This is a contradiction to (\ref{jac2}).	Therefore, there is a non-trivial character on $R_{s,2k-4s+1}(F)$ such that the twisted Jacquet module of $J_{R_{s^2,2k+1},\psi_1\psi_{2,t}}(\Theta_{2k+1})$ with respect to $R_{s,2k-4s+1}(F)$ and this character is non-zero. 
	
	We continue by the same argument repeatedly. For each step, the corresponding unipotent radical acts trivially due to Corollary 3.34 of \cite{CAI19}. Eventually, we obtain a non-zero twisted Jacquet module of  $\Theta_{2k+1}$ with respect to some unipotent orbit that is not comparable to $\calO(\Theta_{2k+1})$. By \hyperref[prop62]{Proposition \ref*{prop62}}, such a twisted Jacquet module must be zero, which is a contradiction. Thus, the action of  $R_{s,2k-2s+1}/Z(R_{s,2k-2s+1})(F)$ on $J_{R_{s,2k+1},\psi_1}(\Theta_{2k+1})$ must be trivial, which completes the proof.
	\end{proof}

With enough tools at our disposal, we are now ready to prove \hyperref[Thm2]{Theorem \ref*{Thm2}}.  Throughout the proof, we identify any subgroup of $\SO_{2k+1}$ or $\SO_{2k+2l+2}$ with its embedded image in $\SO_{4k+2l+3}$ via \ref{ebd}. Similar to the proof of \hyperref[Thm1]{Theorem \ref*{Thm1}}, we only discuss the case of $n_i-n_{i+1} =1$ for all $i =1,2,\cdots, p-1$ and $n_p=1$ for notational simplicity.
\begin{proof}
	Fix a non-zero $\Psi_1 \in \Hom\left(\Theta_{4k+2l+3}, \pi\otimes \Theta(\pi)\right)$. It is clear that $\Psi_1$ must be surjective because both $\pi$ and $\Theta(\pi)$ are irreducible so that $\pi\otimes \Theta(\pi)$ is irreducible as a representation of the group $\wt{\SO}_{2k+1}(F)\times \wt{\SO}_{2k+2l+2}(F)$. As a result, $\Psi_1$ factors through the non-zero twisted Jacquet module $J_{V_{\calO'},\psi_{\calO'}}(\Theta(\pi)) $. In other words, $\Psi_1$ induces a non-zero $\wt{\SO}_{2k+1}(F)\times M^{\psi_{\calO'}}(\calO')(F)V_{\calO'}(F)$-equivariant morphism which we still denote by $\Psi_1$:
	\begin{equation}\label{eq62}
			\Psi_1: \Theta_{4k+2l+3}\rightarrow  \pi\otimes J_{V_{\calO'},\psi_{\calO'}}(\Theta(\pi)).
	\end{equation}
	Here, $M^{\psi_{\calO'}}(\calO')(F)$ denotes the stabilizer of $\psi_{\calO'}$ in $M(\calO')(F)$.
	
	 Recall we have $V_{s_1}=V_{\calO'}\cap R_{s_1}$ where $R_{s_1} =R_{s_1,4k+2l+3}$ is the unipotent radical of the maximal parabolic subgroup of $\SO_{4k+2l+3}$ with Levi subgroup $ \GL_{s_1}\times \SO_{4k+2l-2s_1+3}$.  Notice that $V_{s_1}(F)$ acts on $J_{V_{\calO'},\psi_{\calO'}}(\Theta(\pi))$ by the character $\psi_1$, and $\Psi_1$ is $V_{s_1}(F)$-equivariant. As a result, $\Psi_1$ must factor through the non-zero twisted Jacquet module of $\Theta_{4k+2l+3}$ with respect to the unipotent subgroup $V_{s_1}(F)$ and the character $\psi_{\calO'}$ restricted on $V_{s_1}(F)$, which we denote $\psi_1$. If we denote the resulting map by $\Psi_1'$, then 
	\begin{equation}\label{eq63}
	\Psi_1': J_{V_{s_1},\psi_1}(\Theta_{4k+2l+3})\rightarrow\pi\otimes J_{V_{\calO'},\psi_{\calO'}}(\Theta(\pi))
 \end{equation}
is non-zero.
	 
	Consider the Heisenberg group $H_{s_1}= V_{s_1}\setminus R_{s_1}$, which may be identified with the subgroup of matrices of the form
\[
H_{s_1}=H_{s_1,4k+2l+3}: =\left\{\begin{pmatrix}
	I_{s_1}&&x&&\ast\\
	&I_{k+l-s_1+1}&&&\\
	&&I_{2k+1}&&x^\ast\\
	&&&I_{k+l-s_1+1}&\\
	&&&&I_{s_1}
\end{pmatrix}\in\SO_{4k+2l+3} :x\in \Mat_{s_1\times (2k+1)}\right\}.
\]
As $Z(H_{s_1})(F)\subset V_{s_1}(F)$ acts trivially on  $J_{V_{s_1},\psi_1}(\Theta_{4k+2l+3})$, we consider the action of the abelian quotient $H_{s_1}/Z(H_{s_1})(F)$ on the twisted Jacquet module $J_{V_{s_1},\psi_1}(\Theta_{4k+2l+3})$.  If there exists a non-zero vector in $J_{V_{s_1},\psi_1}(\Theta_{4k+2l+3})$ on which  $H_{s_1}/Z(H_{s_1})(F)$ acts by a character $\psi'$, then $\Psi_1'$ must factor through the non-trivial twisted Jacquet module
\begin{equation}\label{jac3}
J_{H_{s_1}/Z(H_{s_1})(F),\psi'}\left(J_{Z(H_{s_1})(F)}(J_{V_{s_1},\psi_1}(\Theta_{4k+2l+3}))\right) \cong  J_{R_{s_1},\psi_{1}\psi'}(\Theta_{4k+2l+3}).
\end{equation}

We may identify the character group of $H_{s_1}/Z(H_{s_1})(F)$ with $\Mat_{s_1\times (2k+1)}(F)$. Under the action of $ \GL_{s_1}(F)\times \SO_{2k+1}(F)$, suppose $\psi'$ corresponds to a matrix in $\Mat_{s_1\times (2k+1)}(F)$ that lies in the same conjugacy class with a matrix that contains a non-isotropic row vector. Then the product of $\psi_1$ and $\psi'$ is a generic character attached to the unipotent orbit associated with the partition $(3^{s_1}1^{4k+2l+3-3s_1})$. By \hyperref[prop62]{Proposition \ref*{prop62}}, the twisted Jacquet module corresponding to this character (\ref{jac3}) is zero.

It remains to examine over those characters on $H_{s_1}/Z(H_{s_1})(F)$ corresponding to a matrix in $\Mat_{s_1\times (2k+1)}(F)$ with totally isotropic row space. Any such $\psi'$  lies in the same conjugacy class with a character $\psi_{\xi_q}$ corresponding to a matrix $\xi_q$ of the form
	\begin{equation}
	\xi_q = \begin{pmatrix}
		I_q&0\\
		0&0
	\end{pmatrix}\in \Mat_{s_1\times (2k+1)}(F), \quad q =0,1,\cdots, s_1.
\end{equation}

For a fixed $q$,  we check on $\Psi_1'$ restricted to the twisted Jacquet module of the form
\[
J_{H_{s_1}/Z(H_{s_1})(F),\psi_{\xi_q}}\left(J_{Z(H_{s_1})(F)}(J_{V_{s_1},\psi_1}(\Theta_{4k+2l+3}))\right) \cong  J_{R_{s_1},\psi_{1,\xi_q}}(\Theta_{4k+2l+3}),
\]
where $\psi_{1,\xi_q} = \psi_1\psi_{\xi_q}$. Consider the elements $w_q, z_q\in \SO_{4k+2l+3}(F)$ of the form 
\[
z_q = \begin{pmatrix}
	I_{s_1}&&&&\\
	&	\mu_q&&&\\
	&	&I_{2k+1-2q}&&\\
	&&&\mu_q^\ast&\\
	&&&&I_{s_1}
\end{pmatrix}, \mu_q=\begin{pmatrix}
	I_q&&-I_q\\
	&I_{k+l+1-s_1-q}&\\
	&&I_q
\end{pmatrix},
\]
and
\[
\omega_q=\begin{pmatrix}
	I_{s_1}&&&&\\
	&\nu_q&&&\\
	&&I_{2k+1-2q}&&\\
	&&&\nu_q^{-1}&\\
	&&&&I_{s_1}
\end{pmatrix}, \nu_q = \begin{pmatrix}
	&&&I_q&\\
	&I_{s_1-q}&&&\\
	I_q&&&&\\
	&&I_{k+l+q+1-3s_1}&&\\
	&&&&I_{s_1-q}
\end{pmatrix}.
\]
Note that $z_0$ and $\omega_0$ are both the identity matrix. 

Notice that the conjugation action of $w_qz_q$ on $\wt{\SO}_{4k+2l+3}(F)$ preserves $R_{s_1}(F)$. As a result, we have 
\begin{equation}\label{eq64}
	J_{R_{s_1},\psi_{1,\xi}}(\Theta_{4k+2l+3})\cong J_{R_{s_1},\psi_{1}}(\Theta^{w_qz_q}_{4k+2l+3}),
\end{equation}
where $\Theta^{w_qz_q}_{4k+2l+3}$ is the representation of the group $\wt{\SO}_{4k+2l+3}(F)$ obtained by pulling back the representation $\Theta_{4r+2l+3}$ via the conjugation by $w_qz_q$ on $\wt{\SO}_{4k+2l+3}(F)$. For any $g\in \wt{\SO}_{4k+2l+3}(F)$ and any function $\theta$ in the representation $\Theta_{4k+2l+3}$, the action of $g$ on $\theta$ is replaced by the action of $(w_qz_q)^{-1}gw_qz_q $ on $\theta$. Moreover, by \hyperref[prop63]{Proposition \ref*{prop63}},
\begin{equation}
	\Psi_1':   J_{R_{s_1},\psi_{1}}(\Theta^{w_qz_q}_{4k+2l+3})\cong J_{R_{s_1},\psi_{1,\xi_q}}(\Theta_{4k+2l+3})\rightarrow\pi\otimes J_{V_{\calO'},\psi_{\calO'}}(\Theta(\pi))
\end{equation}
factors through the Jacquet module $J_{R_{s^2_1},\psi_{1}}(\Theta^{w_qz_q}_{4k+2l+3})$ where we recall that $R_{s^2_1} =R_{s_1}R_{s_1,4k+2l-2s_1+3}$ and $R_{s_1,4k+2l-2s_1+3}$  is the unipotent radical of the maximal parabolic subgroup of $\SO_{4k+2l-2s_1+3}$ with Levi subgroup $\GL_{s_1} \times \SO_{4k+2l-4s_1+3}$. If we denote the resulting map by $\Psi_1^{''}$, then
\begin{equation}\label{661}
\Psi_1^{''}: 	 J_{R_{s^2_1},\psi_{1}}(\Theta^{w_qz_q}_{4k+2l+3})\rightarrow\pi\otimes J_{V_{\calO'},\psi_{\calO'}}(\Theta(\pi))
\end{equation}
is non-zero.

Let $V_{\calO'}^1 = V_{\calO'}\cap \SO_{2k+2l-2s_1+2}$. The intersection $R_{s_1,4k+2l-2s_1+3} \cap w_qz_q V_{\calO'}^1(w_qz_q)^{-1}$ is non-trivial as long as $q<s_1$. It contains the one parameter subgroup $\{x_\beta(r): r\in F\}$ associated to $\beta$ given by (\ref{root}). The root group $x_\beta(r)$ acts trivially on $J_{R_{s^2_1},\psi_{1}}(\Theta^{w_qz_q}_{4k+2l+3})$, but acts by the non-trivial character $\psi_{\calO'}$ on $J_{V_{\calO'},\psi_{\calO'}}(\Theta(\pi))$. This  implies that $\Psi_1^{''}$ must be zero when $q<s_1$.

It remains only the case of $q=s_1$. Note that $\psi_\xi =\psi_{\xi_{s_1}}$ is generic. Recall that the center $Z(H_{s_1}(F))$ acts trivially on $\Theta_{4k+2l+3}$. Hence, $J_{Z(H_{s_1})}(\Theta_{4k+2l+3}) \cong\Theta_{4k+2l+3}$. By Proposition 5.12(d) of \cite{BernsteinZelevinsky76} or Lemma A.1 of \cite{LE17}, there exists a short exact sequence of $Q_{s_1}(F)\times \wt{\SO}_{2k+1}(F)$-modules
\begin{equation}\label{eq65}
	0\rightarrow \ind_{ Q_{s_1}\times P^0_{s_1}}^{Q_{s_1}\times \wt{\SO}_{2k+1}}\left(J_{H_{s_1},\psi_{\xi}}(\Theta_{4k+2l+3})\right)\rightarrow J_{Z(H_{s_1})}(\Theta_{4k+2l+3})\rightarrow  J_{H_{s_1}}(\Theta_{4k+2l+3})\rightarrow 0,
\end{equation}
where $Q_{s_1}= (\wt{\GL}_{s_1}\times \wt{\SO}_{2k+2l-2s_1+2})V_{s_1}$ is a maximal parabolic subgroup of $\wt{\SO}_{2k+2l+2}$. The product $Q_{s_1}(F)\times \wt{\SO}_{2k+1}(F)$ is the normalizer of $H_{s_1}(F)$ in $\wt{\SO}_{4k+2l+3}(F)$, while the stablizer of $\psi_\xi$ in $Q_{s_1}(F)\times \wt{\SO}_{2k+1}(F)$ is
\[
 (\wt{\GL}_{s_1}^\Delta\times \wt{\SO}_{2k+2l-2s_1+2})V_{s_1}\times P^0_{s_1}\cong  Q_{s_1}\times P^0_{s_1},
\]
where $\wt{\GL}_{s_1}$ is diagonally embeded into the Levi subgroups of the two parabolic subgroups $Q_{s_1} $ and $P_{s_1}$. Here, $\ind_{ Q_{s_1}\times P^0_{s_1}}^{Q_{s_1}\times \wt{\SO}_{2k+1}}$ is the induction with compact support as in \cite{BernsteinZelevinsky77}. 


The two functors  $J_{V_{s_1},\psi_1}$ and $ \ind_{ Q_{s_1}\times P^0_{s_1}}^{Q_{s_1}\times \wt{\SO}_{2k+1}}$ satisfy the following relation
\[
J_{V_{s_1},\psi_1}\circ \ind_{Q_{s_1}\times P^0_{s_1}}^{Q_{s_1}\times\wt{\SO}_{2k+1}} \cong \ind_{\wt{\GL}_{s_1}\times (Q'_{s_1})^0\times P^0_{s_1}}^{\wt{\GL}_{s_1}\times (Q'_{s_1})^0\times\wt{\SO}_{2k+1}}\circ J_{V_{s_1},\psi_1}.
\]
Here, $\wt{\GL}^\Delta_{s_1}(F)\times (Q'_{s_1})^0(F)$ is the stablizer of $\psi_1$ in the Levi factor $\wt{\GL}_{s_1}(F)\times \wt{\SO}_{2k+2l-2s_1+2}(F)$, where $Q'_{s_1}$ is the parabolic subgroup of $ \wt{\SO}_{2k+2l-2s_1+2}$ with Levi subgroup  $\wt{\GL}_{s_1}\times \wt{\SO}_{2k+2l-4s_1+2}$.  Hence, 
\begin{equation}\label{eq661}
J_{V_{s_1},\psi_1}\left(\ind_{ Q_{s_1}\times P^0_{s_1}}^{Q_{s_1}\times \wt{\SO}_{2k+1}}\left(J_{H_{s_1},\psi_{\xi}}(\Theta_{4k+2l+3})\right) \right)
\end{equation}
is isomorphic to 
\begin{equation}\label{eq662}
\ind_{\wt{\GL}^\Delta_{s_1}\times (Q'_{s_1})^0\times P^0_{s_1}}^{\wt{\GL}_{s_1}\times (Q'_{s_1})^0\times\wt{\SO}_{2k+1}}\left(J_{R_{s_1},\psi_{1,\xi}}(\Theta_{4k+2l+3})\right).
\end{equation}

Moreover, $w_{s_1}z_{s_1}$ acts on $\Theta_{4k+2l+3}$ and preserves $R_{s_1}$. By (\ref{eq64}), we further deduce that (\ref{eq662}) is isomorphic to
\begin{equation}\label{eq663}
\ind_{\wt{\GL}^\Delta_{s_1}\times\wt{\SO}_{2k+2l-2s_1+2}\times P^0_{s_1}}^{\wt{\GL}_{s_1}\times \wt{\SO}_{2k+2l-2s_1+2}\times\wt{\SO}_{2k+1}}\left(J_{R_{s_1},\psi_1}(\Theta^{w_{s_1}z_{s_1}}_{4k+2l+3})\right).
\end{equation}
Hereafter, we denote $\ind_{\wt{\GL}^\Delta_{s_1}\times\wt{\SO}_{2k+2l-2s_1+2}\times P^0_{s_1}}^{\wt{\GL}_{s_1}\times \wt{\SO}_{2k+2l-2s_1+2}\times\wt{\SO}_{2k+1}}$ by $\ind_{P^0_{s_1}}^{\wt{\SO}_{2k+1}}$ for simplicity.

Applying the functor  $J_{V_{s_1},\psi_1}$  to (\ref{eq65}), we obtain the short exact sequence
\[
	0\rightarrow \ind_{P^0_{s_1}}^{\wt{\SO}_{2k+1}}\left(J_{R_{s_1},\psi_1}(\Theta^{w_{s_1}z_{s_1}}_{4k+2l+3})\right)\rightarrow J_{V_{s_1},\psi_1}(\Theta_{4k+2l+3})\rightarrow  J_{R_{s_1},\psi_1}(\Theta_{4k+2l+3})\rightarrow 0.
\]
We have just shown that, corresponding to $q =0$,
\[
\Hom_{\wt{\SO}_{2k+1}\times M^{\psi_{\calO'}}(\calO')V_{\calO'}} \left(J_{R_{s_1},\psi_{1}}(\Theta_{4k+2l+3}),\pi\otimes J_{V_{\calO'},\psi_{\calO'}}(\Theta(\pi))\right) = 0.
\]
Then, by (\ref{eq63}), (\ref{eq65}) and (\ref{eq663}), we have
	\begin{equation}\label{eq66}
\Hom_{\wt{\SO}_{2k+1}\times M^{\psi_{\calO'}}(\calO')V_{\calO'}}\left(\ind_{P^0_{s_1}}^{ \wt{\SO}_{2k+1}}\left(J_{R_{s_1},\psi_1}(\Theta^{w_{s_1}z_{s_1}}_{4k+2l+3})\right), \pi\otimes J_{V_{\calO'},\psi_{\calO'}}(\Theta(\pi))\right) \neq 0.
	\end{equation}
By \hyperref[prop63]{Proposition \ref*{prop63}}, we further deduce that 
	\begin{equation}\label{eq665}
	\Hom_{\wt{\SO}_{2k+1}\times M^{\psi_{\calO'}}(\calO')V}\left(\ind_{P^0_{s_1}}^{ \wt{\SO}_{2k+1}}\left(J_{R_{s^2_1},\psi_1}(\Theta^{w_{s_1}z_{s_1}}_{4k+2l+3})\right), \pi\otimes J_{V_{\calO'},\psi_{\calO'}}(\Theta(\pi))\right) \neq 0.
\end{equation}

Continue the same argument by replacing (\ref{eq62}) by any non-zero $\Psi_2$ in the $\Hom$-space (\ref{eq665}). Recall the unipotent subgroup $R_{s_2}$ defined by (\ref{rs2}) and the Heisenberg quotient $H_{s_2} =V_{s_2}\setminus R_{s_2}$. Consider the twisted Jacquet modules with respect to $H_{s_2}(F)$ whose center acts trivially on $J_{R_{s^2_1},\psi_1}(\Theta^{w_{s_1}z_{s_1}}_{4k+2l+3})$. We note that the normalizer of $H_{s_2}$ in $ \wt{\SO}_{2k+1}$ is $\wt{\SO}_{2k-2s_1+1} \subset P^0_{s_1}$. By a similar argument, any non-zero homomorphism
\[
\Psi_2: \ind_{P^0_{s_1}}^{ \wt{\SO}_{2k+1}}\left(J_{R_{s^2_1},\psi_1}(\Theta^{w_{s_1}z_{s_1}}_{4k+2l+3})\right)\rightarrow \pi\otimes J_{V_{\calO'},\psi_{\calO'}}(\Theta(\pi))
\]
must factor through the non-zero twisted Jacquet module with respect to a generic character on $H_{s_2}(F)$, which is 
\[
\ind_{P^0_{s_1}}^{ \wt{\SO}_{2k+1}}\left(\ind_{P^0_{s_2}}^{ \wt{\SO}_{2k-2s_1+1}}J_{R_{s_2},\psi^\circ_2}\left( J_{R_{s^2_1},\psi_1}(\Theta^{w_{s_2}z_{s_2}}_{4k+2l+3})\right)\right).
\] 
Here, $\psi_2^\circ$  is $ \psi_{\calO'}$ restricted to $V_{s_2}$, and $P_{s_2}$ is the standard maximal parabolic subgroup of $\wt{\SO}_{2k-2s_1+1}$ with Levi subgroup $\wt{\GL}_{s_2}\times \wt{\SO}_{2k-2(s_1+s_2)+1}$.

By the transitivity of induction
\[
\ind_{P^0_{s_1}}^{ \wt{\SO}_{2k+1}}\ind_{P^0_{s_2}}^{ \wt{\SO}_{2k-2s_1+1}} \cong \ind_{P^0_{s_2}}^{\wt{\SO}_{2k+1}}.
\]
By \hyperref[prop63]{Proposition \ref*{prop63}}, we obtain that 
	\begin{equation}\label{eq666}
	\Hom_{\wt{\SO}_{2k+1}\times M^{\psi_{\calO'}}(\calO')V_{\calO'}}\left(\ind_{P^0_{s_2}}^{\wt{\SO}_{2k+1}}\left(J_{R_{s^2_2},\psi_2}(\Theta^{w_{s_2}z_{s_2}}_{4k+2l+3})\right), \pi\otimes J_{V_{\calO'},\psi_{\calO'}}(\Theta(\pi))\right)
\end{equation}
is non-zero. Continuing the same argument repeatedly, we obtain that
	\begin{equation}\label{eq67}
	\Hom_{\wt{\SO}_{2k+1}\times M^{\psi_{\calO'}}(\calO')V_{\calO'}}\left(\ind_{\wt{\SO}_l}^{\wt{\SO}_{2k+1}}\left(J_{R_{s^2_p},\psi_p}(\Theta^{w_lz_l}_{4k+2l+3})\right), \pi\otimes J_{V_{\calO'},\psi_{\calO'}}(\Theta(\pi))\right) \neq 0.
\end{equation}
Here, $R_{s^2_p}$ is the unipotent radical of the standard maximal parabolic subgroup of $\SO_{4k+2l+3}$ with Levi part 
\[
\GL^{2(n_1-n_2)}_{s_1}\times \GL^{2(n_2-n_3)}_{s_2}\times \cdots\times \GL^{2(n_{p-1}-n_p)}_{s_{p-1}}\times \GL^{2n_p}_l\times \SO_{4l+1}.
\]
The character $\psi_p:R_{s^2_p}(F)\rightarrow \bbC^\times$ is the product of $\psi_1$ and the characters corresponding to the non-zero twisted Jacquet modules in each of the repeated steps.

We now proceed to the last step. Let $R_p $ be the unipotent radical of the maximal parabolic subgroup of $\SO_{4k+1}$ with Levi part $\GL_l\times \SO_{2l+1}$. Consider the action of the abelian quotient $H_p(F) = \omega_lz_l V^p_{\calO'}(\omega_lz_l)^{-1}\setminus R_p$ on \begin{equation}\label{jac6}
\ind_{\wt{\SO}_lU}^{\wt{\SO}_{2k+1}}\left(J_{R_{s^2_p},\psi_p}(\Theta^{w_lz_l}_{4k+2l+3})\right).
\end{equation}
If there is any non-zero vector in (\ref{jac6}) on which $H_p(F)$ act by a character, then any non-zero $\Psi_p$ in (\ref{eq67}) must factor through the non-zero twisted Jacquet module of (\ref{jac6}) with respect to $H_p(F)$ and the corresponding character. Following the same argument as in the proof of \hyperref[Thm1]{Theorem \ref*{Thm1}}, we only need to consider  characters $\psi_\xi$ on $H_p/Z(H_p)(F)$ such that the product of $\psi_{\calO'}$ and $\psi_\xi$ restricted to $R_p(F)$ is non-generic.  That is, the matrix in $\Mat_{l\times (2l+1)}(F)$ corresponding to $\psi_{\calO'}\psi_\xi$ has totally isotropic row space. We deduce that such a character $\psi_\xi$ must correspond to a matrix given by 
	\[
	\xi =\begin{pmatrix}
		\lambda_1\\
		\vdots\\
		\lambda_l
	\end{pmatrix} \in\Mat_{l\times l}(F)
	\]
satisfying the following conditions:
 \begin{enumerate}
	\item Each $\lambda_i$ except $\lambda_{\frac{l+1}{2}}$ is non-zero isotropic. \\
	\item Each pair $\lambda_i$ and $\lambda_{l+1-i}$ for $i=1,2,\cdots, \frac{l-1}{2}$ is a hyperbolic pair, i.e.
	\[
	(\lambda_i,\lambda_{l+1-i}) = 1,\;i=1,2,\cdots,\frac{l-1}{2}.
	\]
	These hyperbolic pairs are mutually orthogonal.
	\item $\lambda_{\frac{l+1}{2}}$ is non-isotropic with unit length. 
\end{enumerate}

The group $\SO_l(F)$ acts transitively on the set of these matrices. This allows us to pick the identity matrix $I_l$ as a representative. Denote the corresponding character on $H_p/Z(H_p)(F)$ by $\psi_{I_l}$. Any non-zero $\Psi_p$ in (\ref{eq67})  factors through the non-zero twisted Jacquet module of (\ref{jac6}) given by
\begin{equation}
	\ind_{\wt{\SO}_l}^{\wt{\SO}_{2k+1}}\left(	\ind_{\wt{\SO}^\Delta_l}^{\wt{\SO}_{l}\times \wt{\SO}_l}\left(J_{R_p,\psi_{V_{\calO'},I_l}}J_{R_{s^2_p},\psi_p}(\Theta^{w_lz_l}_{4k+2l+3})\right)\right) \cong \ind_{U_\calO}^{\wt{\SO}_{2k+1}}\left(J_{R,\psi_{V_{\calO'},p,I_l}}(\Theta^{w_lz_l}_{4k+2l+3})\right).
\end{equation}
We note that the normalized induction $\ind_{U_\calO}^{\wt{\SO}_{2k+1}}$ is $\ind_{L^\Delta(\calO)\times U_\calO}^{L(\calO)\times \wt{\SO}_{2k+1}}$ where
\[
L(\calO) \cong \wt{\GL}_{s_1}^{n_1-n_2} \times \wt{\GL}_{s_2}^{n_2-n_3}\times \cdots \times \wt{\GL}_{l}^{n_p}\times\wt{\SO}_l ,
\]
and
\[
L(\calO)^{\Delta} \cong \left(\wt{\GL}^\Delta_{s_1}\right)^{n_1-n_2}\times\left(\wt{\GL}^\Delta_{s_2}\right)^{n_2-n_3}\times \cdots \times \left(\wt{\GL}^\Delta_{l}\right)^{n_p}\times \wt{\SO}^\Delta_l.
\]
 Thus, we conclude that 
	\begin{equation}\label{eq68}
	\Hom_{\wt{\SO}_{2k+1}\times M^{\psi_{\calO'}}(\calO')V_{\calO'}}\left(\ind_{U_\calO}^{\wt{\SO}_{2k+1}}\left(J_{R,\psi_{V_{\calO'},p,I_l}}(\Theta^{w_lz_l}_{4k+2l+3})\right), \pi\otimes J_{V_{\calO'},\psi_{\calO'}}(\Theta(\pi))\right) \neq 0.
\end{equation}

We now examine the action of $U_\calO(F)$ on (\ref{eq68}). The group $U_\calO(F)$ acts on the left by the character $\psi_\calO$, since $(w_lz_l)^{-1}U_\calO w_lz_l \subset R$ and $\psi_{V_{\calO'},p,I_l}$ is defined on $R(F)$. As a result, there exists a non-zero vector in $\pi$ on which $U_\calO(F)$ acts by the same character $\psi_\calO$, and
	\begin{equation}\label{eq69}
	\Hom_{M^{\psi_\calO}(\calO)\times M^{\psi_{\calO'}}(\calO')}\left(J_{R,\psi_{V_{\calO'},p,I_l}}(\Theta^{w_lz_l}_{4k+2l+3}), J_{U_\calO,\psi_\calO}\left(\pi\right)\otimes J_{V_{\calO'},\psi_{\calO'}}(\Theta(\pi))\right) \neq 0.
\end{equation}
Thus, $J_{U_\calO,\psi_\calO}\left(\pi\right)$ is non-zero.  This completes the proof. 
	\end{proof}
\bibliographystyle{plain}
\bibliography{Bibtexfile}{}
\end{document}